\documentclass{amsart}

\usepackage{hyperref}

\usepackage{pdfrender,xcolor}

\hypersetup{colorlinks=true,
	linkcolor=black, 
	urlcolor=blue} 
\usepackage{amsmath,amssymb,amsfonts,amsthm}
\usepackage{amsthm}
\usepackage{enumerate}
\numberwithin{equation}{section}
\usepackage{marginnote}
\usepackage[a4paper]{geometry}

\reversemarginpar

\usepackage[english]{babel}
\usepackage[babel, english=american]{csquotes}

\newtheoremstyle{Teorema}
{3pt}
{3pt}
{\slshape}
{}
{\bfseries}
{:}
{\newline}
{}

\newtheorem{theorem}{Theorem}[section]
\newtheorem{definition}[theorem]{Definition}

\newtheorem{corollary}[theorem]{Corollary}

\newtheorem{lemma}[theorem]{Lemma}

\theoremstyle{definition}

\DeclareMathOperator{\D}{d}        
\DeclareMathOperator{\supp}{supp}  

\newcommand{\ZZ}{\mathbb{Z}}      
\newcommand{\RR}{\mathbb{R}}      
\newcommand{\CC}{\mathbb{C}}      
\newcommand{\n}{\vec{n}}          

\date{\today}

\begin{document}
	
\pdfrender{StrokeColor=black,TextRenderingMode=2,LineWidth=0.5pt}
	
	\title[Strong asymptotics of polynomials]{Strong asymptotics of multi-level Hermite-Pad\'{e} polynomials}
	
	\author{L. G. Gonz\'alez Ricardo}
	\address{Department of Mathematics, Universidad Carlos \textsc{iii} de Madrid, Avenida de la Universidad, 30 CP-28911, Leganés, Madrid, Spain.}
	\email{luisggon@math.uc3m.es}
	\thanks{The first author was supported in part by a research fellowship from Department of Mathematics of Universidad Carlos III de Madrid, Spain.}
	
	\author{G. L\'opez Lagomasino}
	\address{Department of Mathematics, Universidad Carlos \textsc{iii} de Madrid, Avenida de la Universidad, 30 CP-28911, Leganés, Madrid, Spain.}
	\email{lago@math.uc3m.es}
	\thanks{The second author was supported in part by the research grant PGC2018-096504-B-C33 of Ministerio de Ciencia, Innovaci\'on y Universidades, Spain.}
	
	\date{\today}
	
	\begin{abstract}
		We obtain the strong asymptotics of multiple orthogonal polynomials which arise in a mixed type Hermite-Pad\'e approximation problem defined on a Nikishin system of functions.  The results obtained allow to give exact estimates of the rate of convergence of the approximating rational functions and the strong asymptotics of Cauchy biorthogonal
polynomials.
\end{abstract}

	\maketitle
	
	\medskip
	
	
	\medskip
	

\section{Introduction}

\subsection{Background}
Szeg\H{o}'s theory of orthogonal polynomials \cite[Chap. X-XII]{szego} is a gem of real and complex analysis not only for its intrinsic beauty but also because of its many consequences, extensions, and applications in different fields (see \cite{simon1,simon2}). The natural framework in which the theory is developed is for polynomials orthogonal with respect to measures supported on the unit circle. Then, it may be transposed to polynomials orthogonal with respect to measures supported on bounded intervals of the real line. To illustrate, we present what is perhaps the most important result of Szeg\H{o}'s theory in the context of orthogonal polynomials on segments of the real line. First, we need some notation.

Let $\mu$ be a finite positive Borel measure supported on an interval $\Delta = [a,b] \subset \RR$. We say that $\mu$ satisfies Szeg\H{o}'s condition and write $\mu \in \mathcal{S}(\Delta)$ if
  \begin{equation}
        \label{df:szego:c}
        \int_{\Delta}   {\ln \mu'(x)} \D \eta_{\Delta}(x)  > -\infty,
\end{equation}
where
\[ \D \eta_{\Delta}(x) = \frac{\D x}{\sqrt{(b-x)(x-a)}}, \qquad x \in \Delta,
\]
and $\mu'$ denotes the Radon-Nikodym derivative of $\mu$ with respect to the Lebesgue measure. When $\mu \in \mathcal{S}(\Delta)$
	\[
		\mathsf{G}(\mu,z) := \exp\left[\frac{\sqrt{(z-b)(z-a)}}{2\pi}\int_\Delta \frac{\ln(\sqrt{(b-x)(x-a)}\,\mu'(x))}{x-z}\D \eta_{\Delta}(x) \right],
	\]
    is called the Szeg\H{o} function of $\mu$. The square root outside the integral is taken to be positive for $z > b$ and those inside the integral are positive when $x \in (a,b)$. The Szeg\H{o} function verifies
	\begin{equation} \label{RN}
		\begin{cases}
			\mathsf{G}(\mu,z) \neq 0 \textrm{ for } z \in \overline{\CC}\setminus\Delta,\\
			\mathsf{G}(\mu,\infty)>0,\\
			\lim_{y \to 0} |\mathsf{G}(\mu,x+iy)|^2 =  (\sqrt{(b-x)(x-a)}\,\mu'(x))^{-1},\quad \textrm{for a.e.}\,\, x \in \Delta.
		\end{cases}
	\end{equation}

    Let $\Psi$ be the conformal map from $\overline{\CC} \setminus \Delta$ onto the complement (in $\overline{\CC}$) of the closed unit disk such that $\Psi(\infty) = \infty$, $\Psi'(\infty) > 0$.
 Let $l_n(z) = \kappa_n z^n + \cdots $ be the $n$-th orthonormal polynomial with respect to $\mu$ with positive leading coefficient; that is, $\kappa_n > 0$ and
    \[ \int l_n(x) l_k(x) \D \mu(x)= \delta_{n,k},  \qquad n,k \geq 0
    \]
    where $\delta_{n,k}$ denotes the Kronecker $\delta$. We have

    \begin{theorem}[G. Szeg\H{o}]
        \label{Szego}
        Assume that $\mu \in \mathcal{S}(\Delta)$, then
        \[ \lim_{n\to \infty} \frac{l_n(z)}{\Psi^n(z)} = \frac{1}{\sqrt{2\pi}} \mathsf{G}(\mu,z),
        \]
        uniformly on compact subsets of $\overline{\CC} \setminus \Delta$, and
        \[ \lim_{n \to \infty} \frac{\kappa_n}{C^n} = \frac{\mathsf{G}(\mu,\infty)}{\sqrt{2\pi}},
        \]
        where $C^{-1} = (b-a)/4$ is the logarithmic capacity of $\Delta$.
    \end{theorem}

    The second formula follows directly from the first one applied at $z = \infty$. The first formula expresses that asymptotically, the $n$-th orthonormal polynomial behaves like the $n$-th power of the fixed function $\Psi$.

    Our goal is to obtain a similar result for the multiple orthogonal polynomials (which share orthogonality relations with a system of measures)  arising in  a certain mixed type approximation problem. Previously, this type of extension of Szeg\H{o}'s theory was considered by A.I. Aptekarev \cite{sasha89, sasha99} for type \textsc{ii} Hermite-Pad\'e polynomials of so called Angelesco and Nikishin systems of measures. Our work is related with Nikishin systems of measures but the orthogonality relations differ from the type \textsc{ii} case. However, as we explain later, our main result Theorem \ref{th1}
    generalizes in several directions that of A.I. Aptekarev in \cite[Theorem 2]{sasha99}.

    In the study of multi-peakon solutions of the Degasperis-Procesi equation, a mixed type Hermite-Pad\'e approximation problem was introduced  in \cite{JacLun05}. The same construction showed to be useful in the analysis of  the two matrix model \cite{BGS13} and Cauchy biorthogonal polynomials \cite{Bertola:CBOPs}. The approximated objects are systems of two functions, defined in terms of Cauchy transforms of two measures supported on disjoint intervals of the real line. This application  motivated the introduction in \cite{LSJ19} of an approximation problem for an arbitrary Nikishin system generated by $m$ measures (which reduces to the previous case when $m=2$) and study the convergence of the method.

    Recently, the approximation problem was generalized further by V.G. Lysov in \cite{lysov20} where he studied the asymptotic zero distribution and the logarithmic asymptotics of the associated multiple orthogonal polynomials. The ratio asymptotics of these multiple orthogonal polynomials was given in \cite{SLL22}. Here, we wish to find their strong asymptotic behavior. Before stating our main result, let us briefly present the so called Nikishin systems of measures and Lysov’s definition of multi-level Hermite-Pad\'e approximants.

\subsection{Nikishin systems and their multi-level Hermite-Pad\'e polynomials}

    Nikishin systems of measures were first introduced by E.M. Nikishin in \cite{nikishin82}. Throughout the years, it has been shown that polynomials which share orthogonality
    relations with respect to such systems of measures behave asymptotically similar  to
    standard orthogonal polynomials (with respect to a single measure). A brief overview of the subject may be found in \cite{lago21}.

    Let $\mathcal{M}(\Delta)$ denote the class of all finite positive Borel measures $s$ with compact support $\supp s \subset \Delta=[a,b]$ containing an infinite set of points. If $s$  satisfies Szeg\H{o}'s condition \eqref{df:szego:c} we write $s \in \mathcal{S}(\Delta)$. By
    \[
        \widehat{s}(z) := \int\frac{\D s(x)}{z-x}
    \]
    we denote the Cauchy transform of the measure $s$. Obviously, $\widehat{s} \in\mathcal{H}(\overline{\CC}\setminus\Delta)$ (holomorphic on $\overline{\CC}\setminus\Delta$).

	Let $\Delta_\alpha$, $\Delta_\beta$ be two intervals contained in the real line such that $\Delta_\alpha\cap\Delta_\beta=\varnothing$. Consider the measures $\sigma_\alpha\in\mathcal{M}(\Delta_\alpha)$ and $\sigma_\beta\in\mathcal{M}(\Delta_\beta)$. With these two measures we construct a third one as follows (using differential notation)
	$$\D\langle \sigma_\alpha,\sigma_\beta\rangle(x): = \widehat{\sigma}_\beta(x)\D\sigma_\alpha(x).$$
Notice that $\langle \sigma_\alpha,\sigma_\beta\rangle \neq \langle \sigma_\beta,\sigma_\alpha \rangle$.  We can extend this operation inductively to more than two measures. For example, if   $\Delta_\gamma \cap\Delta_\alpha = \varnothing$ and
$\sigma_\gamma \in \mathcal{M}(\Delta_\gamma)$, we define
\[ \langle \sigma_\gamma,\sigma_\alpha,\sigma_\beta \rangle :=  \langle \sigma_\gamma, \langle \sigma_\alpha,\sigma_\beta \rangle \rangle.
\]
	
	\begin{definition}
		\label{df:nikishin}
		Take a collection $\Delta_j$, $j=1,\ldots,m, m \geq 2,$ of intervals such that
\begin{equation} \label{empty}  \Delta_j\cap\Delta_{j+1}=\varnothing, \qquad  j=1,\ldots,m-1.
\end{equation}
		Let $(\sigma_1,\ldots,\sigma_m)$ be a system of measures such that  $\sigma_j\in\mathcal{M}(\Delta_j)$, $j=1,\ldots,m$. We say that $(s_{1,1},\ldots,s_{1,m})=\mathcal{N}(\sigma_1,\ldots,\sigma_m)$, where
		\[  s_{1,1}=\sigma_1,\; s_{1,2}=\langle\sigma_1,\sigma_2 \rangle,\;s_{1,3}=\langle\sigma_1,\langle \sigma_2,\sigma_3\rangle \rangle,\, \ldots, \; s_{1,m}=\langle \sigma_1, \langle \sigma_2,\ldots,\sigma_m \rangle\rangle
		\]
		is the Nikishin system of measures generated by $(\sigma_1,\ldots,\sigma_m)$. The vector $(\widehat{s}_{1,1},\ldots,\widehat{s}_{1,m})$ is called the Nikishin system of functions.
	\end{definition}
	
	Notice that any sub-system of $(\sigma_1,\ldots,\sigma_m)$ of consecutive measures is also a generator of some Nikishin system. In the sequel, for $1\leq j\leq k\leq m$, we will write
	\[  s_{j,k} := \langle \sigma_j, \sigma_{j+1},\ldots, \sigma_k \rangle,\qquad s_{k,j} := \langle \sigma_k, \sigma_{k-1},\ldots, \sigma_j \rangle.
	\]
Without further notice, throughout the paper we consider that $m \geq 2$.

    Let $(\ZZ_+^m)^*$ be the set of all $m$-dimensional vectors with non-negative integer components, not all equal to zero. For $\n = (n_1,\ldots,n_{m})\in (\ZZ_+^m)^*$ we define $|\n|= n_1 + \cdots + n_{m}$ and $\eta_{\n,j} = n_1+\cdots+n_j$.
    \begin{definition}
        \label{df:HP:ly}
        Consider the Nikishin system $\mathcal{N}(\sigma_1,\ldots,\sigma_m)$ and $\n = (n_1,\ldots,n_{m}) \in (\ZZ_+^m)^*$. There exist polynomials $a_{\n,0}, a_{\n,1},\ldots,a_{\n,m}$, where $\deg a_{\n,j}\leq |\n|-1$, $j=0,1,\ldots,m-1,$ and $\deg a_{\n,m}\leq |\n|$, not all identically equal to zero, called multi-level (ML) Hermite-Padé polynomials of $\mathcal{N}(\sigma_1,\ldots,\sigma_m)$ with respect to $\n$, that verify
        \begin{equation}
            \label{df:HP:ly:1}
            \mathcal{A}_{\n,j}(z) := \left((-1)^j a_{\n,j} + \sum_{k=j+1}^m (-1)^k a_{\n,k}\widehat{s}_{j+1,k}\right)(z) = \mathcal{O}\left(\frac{1}{z^{n_{j+1}+1}}\right), \qquad z \to \infty,
        \end{equation}
        where $j=0,\ldots,m-1$ (the asymptotic expansion of $\mathcal{A}_{\n,j}$ at $\infty$ begins with $z^{-n_{j+1} -1}$, or higher). For completeness, set $\mathcal{A}_{\n,m} := (-1)^m a_{\n,m}$.
    \end{definition}

    We warn the reader that with our terminology in \cite[Problem A]{lysov20} the ML Hermite-Padé polynomials were defined with respect to the system $\mathcal{N}(\sigma_m,\ldots,\sigma_1)$. Initially, in \cite{LSJ19} ML Hermite-Pad\'e approximants were defined only for multi-indices of the form $(n,0,\ldots,0)$.

    Following Mahler's denomination \cite{mahler}, a multi-index $\n \in (\ZZ_+^m)^*$ is said to be normal if $\deg a_{{\n},j} = |\n| -1$, $j=0,\ldots,m-1,$ and $\deg a_{{\n},m} = |\n|$. The system of functions is said to be perfect when all the multi-indices are normal. In \cite[Theorem 1.1]{lysov20}, it was proved that Nikishin systems of functions are perfect for this approximation problem. Normality implies that the vector $(a_{\n,0},\ldots,a_{\n,m})$ is uniquely determined up to a multiplicative factor. In the sequel, we normalize this vector so that $a_{\n,m}$ has leading coefficient equal to one.

    The zeros of the forms $\mathcal{A}_{\n,j}$ have some interesting properties (see \cite[Lemma 2.2]{SLL22} or Subsection \ref{2.1} below). The form $\mathcal{A}_{\vec n,0}$ has no zero in $\CC \setminus \Delta_1$. For $j=1,\ldots, m,$ $\mathcal{A}_{\vec n,j}$ has exactly $\eta_{\n,j}$ zeros in $\CC \setminus \Delta_{j+1}$, $(\Delta_{m+1} = \varnothing )$, they are all simple and lie in $\mathring{\Delta}_j$ (the interior of $\Delta_j$ with the Euclidean topology of $\RR$). If $Q_{\vec n,j}$, $j= 1,\ldots,m,$ denotes the monic  polynomial whose roots are the simple zeros which $\mathcal{A}_{\vec n,j}$ has in $\mathring{\Delta}_j$, then
    \[
       	\frac{\mathcal{A}_{\n,j}(z)}{Q_{\n,j}(z)} = \mathcal{O}\left(\frac{1}{z^{\eta_{\n,j+1}+1}}\right) \in \mathcal{H}(\CC \setminus\Delta_{j+1}), \qquad z\to \infty, \qquad j=1,\ldots,m-1.
    \]
    For each $j = 0,\ldots,m-1$ the order of interpolation at infinity prescribed in \eqref{df:HP:ly:1} is exact.

    A sequence of distinct multi-indices $\Lambda \subset (\ZZ_+^m)^*$ is said to be a ray sequence if
    \[
        \lim_{\n \in \Lambda} n_j/|\n| = p_j, \qquad j=1,\ldots,m.
    \]
    (Consequetly, $\sum_{j=1}^m p_j= 1$.)
    The main results in \cite[Theorem 1.2, Corollary 1.1]{lysov20} give the asymptotic zero distribution and the logarithmic asymptotics of ray sequences of the polynomials $Q_{\n,j}$, $j=1,\ldots.m,$ when $\sigma_j' \neq 0$, $j=1,\ldots,m$ (this assumption can be relaxed to requiring that the measures $\sigma_j$ be regular, see \cite[Chap. 3]{Stahl_Totik} for the definition of regular measure). With the assumptions in Lysov's paper it is possible to prove ratio asymptotics (see \cite[Theorem 3.4, Corollary 3.5]{SLL22}).

     $\Lambda \subset (\ZZ_+^m)^*$ is said to be a straight ray sequence when
    \begin{equation}
        \label{const_ray}
        \frac{n_j}{|\n|}= p_j, \qquad j=1,\ldots,m, \qquad   \n\in\Lambda.
    \end{equation}
    Our goal is to give the strong asymptotics of sequences  of polynomials $Q_{\n,j}$, $a_{\n,j-1}$, $j=1,\ldots,m,$ and forms $\mathcal{A}_{\n,j}$, $j=0,\ldots,m$, for $\n$ belonging to a straight  ray sequence of multi-indices assuming that $\sigma_j \in \mathcal{S}(\Delta_j)$, $j=1,\ldots,m$.

\subsection{Statement of the main results}

    Let $ (p_1,\ldots,p_m) \in \RR_{\geq 0}^m$, $p_1 > 0.$  Set
    \begin{equation}
        \label{Ps}
        P_j = \sum_{k=1}^{j}p_k, \qquad j=1,\ldots,m.
    \end{equation}
    Notice that $P_j > 0$, $j=1,\ldots,m$. Let $\vec{\Delta} = (\Delta_1,\ldots,\Delta_m)$ be a collection of bounded intervals  $\Delta_j=[a_j,b_j]$  of the real line which satisfies \eqref{empty}. By $\mathcal{M}_1(\vec{\Delta})$ we denote the cone of all vector probability measures   $\vec{\mu} = (\mu_1,\ldots,\mu_m)$ such that $\supp (\mu_j) \subset \Delta_j$,  $j=1,\ldots,m$.

    Given a finite Borel measure $\mu$ with compact support its logarithmic potential is defined as
    \[ V^\mu(z) := \int\ln\frac{1}{|z-x|}\D \mu (x).
    \]
    In the present paper, we need the solution $\vec{\lambda} = (\lambda_1,\ldots,\lambda_m)\in \mathcal{M}_1(\vec{\Delta})$ of the   vector equilibrium problem
    \begin{equation}
        \label{eq:vector}
        V^{\lambda_j}(x) - \frac{1}{2}\left[ \frac{P_{j-1}}{P_j}\,V^{\lambda_{j-1}}(x) +  \frac{P_{j+1}}{P_j}\,V^{\lambda_{j+1}}(x)\right]
        \begin{cases}
            =\omega_j, & x\in\supp \lambda_j,\\
            \geq \omega_j, & x\in \Delta_j,
        \end{cases}
    \end{equation}
    for $j=1,\ldots,m$. Here, $V^{\lambda_{m+1}}\equiv V^{\lambda_0} \equiv 0$, $P_0 = P_{m+1} = 0$.

    The problem \eqref{eq:vector} is equivalent to
    \begin{equation}
    	\label{eq:vector*}
    	P_j^2 V^{\lambda_j}(x) - \frac{1}{2}\left[ P_jP_{j-1}\,V^{\lambda_{j-1}}(x) +  P_jP_{j+1}\,V^{\lambda_{j+1}}(x)\right]
        \begin{cases}
    	    =  P_j^2\omega_j, & x\in\supp \lambda_j,\\
    	  \geq P_j^2\omega_j, & x\in \Delta_j,
        \end{cases}
    \end{equation}
    for $j=1,\ldots,m$. It is well known that there exists a unique $\vec{\lambda} \in \mathcal{M}_1(\vec{\Delta})$ such that \eqref{eq:vector*} holds for each $j=1,\ldots,m$.      $\vec{\lambda}$ is called the unitary vector equilibrium measure and $\vec{\omega} := (\omega_1,\ldots,\omega_m)$ the vector equilibrium constant corresponding to the vector equilibrium problem with interaction matrix
    \begin{equation}
        \label{interaction}
        \mathcal{C} := \left(
        \begin{array}{ccccc}
            P_1^2 & -\frac{P_1 P_2}{2} & 0 & \cdots & 0 \\
            -\frac{P_2 P_1}{2} & P_2^2 & -\frac{P_2 P_3}{2} & \ddots & 0 \\
            0 & -\frac{P_3P_2}{2} & P_3^2  & \ddots & 0 \\
            \vdots & \ddots & \ddots & \ddots & \vdots \\
            0 & 0 & 0 & \cdots & P_m^2
        \end{array} \right)
    \end{equation}
    on the system of intervals $\vec{\Delta}$.

    The proof of the existence of unique $\vec{\lambda}$ and $\vec{\omega}$ in the present situation follows from results contained in \cite[Chap. 5, \S 4]{NS} (see also \cite[Section 4]{bel}) due to the fact that $\mathcal{C}$ is a symmetric positive definite matrix whose entries $c_{j,k}$ are non-negative whenever $\Delta_j \cap \Delta_k \neq \varnothing$. We  will need that $\supp (\lambda_j) = \Delta_j$, $j=1,\ldots,m$, in order to have equality in \eqref{eq:vector} throughout all $\Delta_j$. It will be shown in Lemma \ref{soportecompleto} that this is true whenever $p_1 > 0$ and $p_1 \geq \cdots \geq p_m$.

    In the sequel,
    \begin{equation}
        \label{compar}
        \Phi_j:= e^{-v_j},\qquad v_j := V^{\lambda_j} + i\widetilde{V}^{\lambda_j}, \qquad C_j = e^{\omega_j},
    \end{equation}
    where $\widetilde{V}^{\lambda_j}$ denotes the harmonic conjugate of $V^{\lambda_j}$ on $\CC\setminus\Delta_j$, $j=1,\ldots,m,$ which equals zero on $ (b_j,+\infty)$. The function $\Phi_j$ is a single-valued analytic function in $\CC \setminus \Delta_j$ (notice that the values of $-v_j$ vary by a multiple of $2\pi i$ as we circulate around $\Delta_j$ along a closed Jordan curve) with a simple pole at $\infty$. These functions will play the same role as $\Psi$ in Szeg\H{o}'s Theorem \ref{Szego}. The functions $\Phi_j^{\eta_{\n,j}}$ will be used to compare  the polynomials $Q_{\n,j}$. We still need to introduce the limiting functions.

    Assume that $\vec{\sigma} = (\sigma_1,\ldots,\sigma_m) \in \mathcal{S}(\vec{\Delta})$; that is, $\sigma_j \in \mathcal{S}(\Delta_j)$, $j=1,\ldots,m$. There exists a unique system of Szeg\H{o} functions $\vec{\mathsf{G}} = (\mathsf{G}_1,\ldots,\mathsf{G}_m)$ (see Theorem \ref{puntofijo} below), where $\mathsf{G}_k \in \mathcal{H}(\overline{\CC} \setminus \Delta_k)$, which verifies the system of boundary value equations
    \begin{equation}
    	\label{boundcond}
    	|\mathsf{G}_j(x)|^2 = \frac{\sqrt{|x -b_{j+1}||x-a_{j+1}|}\,\mathsf{G}_{j-1}(x) \mathsf{G}_{j+1}(x)}{\sqrt{(b_j -x)(x-a_j)} \sigma_j'(x)},\quad \mbox{a.e. on}\,\, \Delta_j, \quad j=1,\ldots,m.
    \end{equation}
    Here, $\mathsf{G}_0 \equiv \mathsf{G}_{m+1} \equiv 1$ and when $j=m$, $\sqrt{|x -b_{m+1}||x-a_{m+1}|}$ is substituted by $1$. Notice that $\mathsf{G}_j(z)$ is the Szeg\H{o} function with respect to a measure on $\Delta_j$ whose Radon-Nikodym devivative is (compare \eqref{boundcond} and \eqref{RN})
    \[  \frac{\sigma_j'(x) }{\sqrt{|x -b_{j+1}||x-a_{j+1}|}\,\mathsf{G}_{j-1}(x)\mathsf{G}_{j+1}(x)}.
    \]

    \begin{theorem}
        \label{th1}
        Let $\Lambda = \Lambda(p_1,\ldots,p_m)$ be a straight ray sequence of multi-indices such that \eqref{const_ray} takes place,  and $p_1 \geq \cdots \geq p_m$. Assume that $\vec{\sigma} \in \mathcal{S}(\vec{\Delta})$.
        Then
        \begin{equation}
 		\label{limfund***}
 		\lim_{\n\in\Lambda} \frac{Q_{\n,j}(z)}{\Phi_j^{\eta_{\n,j}}(z)} = \frac{\mathsf{G}_j(z)}{\mathsf{G}_j(\infty)}, \qquad j=1,\ldots,m,
 	\end{equation}
        uniformly on each compact subset of $\overline{\CC} \setminus \Delta_j$,  where $\vec{\mathsf{G}} = (\mathsf{G}_1,\ldots,\mathsf{G}_m)$ is the vector of Szeg\H{o} functions whose components are determined by the boundary conditions  \eqref{boundcond} and the $\Phi_j$ are as defined in \eqref{compar} constructed from the vector equilibrium problem associated with the given vector $(p_1,\ldots,p_m)$. The normalizing constants $\kappa_{\n,j}$, $j=1,\ldots,m,$  defined in \eqref{def:kappa} verify
        \begin{equation}
        	\label{kappas}
        	\lim_{n \in \Lambda} \frac{\kappa_{\n,j}}{C_j^{\eta_{\n,j}}} = \frac{\mathsf{G}_j(\infty)}{\sqrt{2\pi}},
        \end{equation}
        where the $C_j$ were given in \eqref{compar}. Additionally, the ML Hermite-Pad\'e polynomials satisfy
        	\begin{equation}
        		\label{MLasin}
        		\lim_{\n\in\Lambda} \frac{a_{\n,j}(z)}{\Phi_m^{|\n|}(z)} = \frac{\mathsf{G}_m(z)}{\mathsf{G}_m(\infty)}\widehat{s}_{m,j+1}(z), \qquad j=0,\ldots,m-1,
	\end{equation}
	 uniformly on compact subsets of $\overline{\CC} \setminus \Delta_m$.
    \end{theorem}

    The asymptotic formulas satisfied by the forms $\mathcal{A}_{\n,j}$, $j=0,\ldots,m-1,$ are more entangled and we give them in Corollary \ref{asym_forms}.

    The proof of Theorem \ref{th1} makes use of fixed point theorems. This method was employed by A.I. Aptekarev in  \cite{sasha89} for the study of the strong asymptotics of type \textsc{ii} Hermite-Pad\'e polynomials of Angelesco systems of two measures and used again in \cite{sasha99}   to obtain the strong asymptotics of type \textsc{ii} Hermite-Pad\'e polynomials of Nikishin systems with an arbitrary number of measures. Here, we adapt the scheme proposed in \cite{sasha99}  to the case of ML Hermite-Pad\'e approximation, simplify some proofs, and correct some statements.

    Type II Hermite-Pad\'e polynomials and ML Hermite-Pad\'e polynomials in general do not coincide, but for the multi-indices considered in Theorem \ref{th1} for which $n_1 \geq n_2 \geq \cdots \geq n_m$ their construction is equivalent as follows from V.G. Lysov's result \cite[Proposition 1.1]{lysov20}. In \cite{sasha99} the measures are
    assumed to be absolutely continuous with respect  to Lebesgue measure and the multi-indices are of the form $\n = (n,n,\ldots,n), n \geq 1$. Therefore, our result covers that of A.I. Aptekarev with weaker assumptions. Theorem \ref{th1} extends Szeg\H{o}'s classical result on the strong asymptotic of orthogonal polynomials to the case when $m > 1$.

    Let $\mathcal{C}({\vec{\Delta}})$ be the linear space of all vector functions $\vec{f} = (f_1,\ldots,f_m)$ such that $f_j$ is continuous on $\Delta_{j-1} \cup \Delta_{j+1}$, $(\Delta_0 = \Delta_{m+1} = \varnothing)$. Set
    \[ \|\vec{f}\|_{\vec{\Delta}} :=  \max_{1\leq j\leq m} \left\{ \| f_j \|_{\Delta_{j-1}\cup \Delta_{j+1}}\right\}
    \]
    where $\|\cdot\|_{\Delta_{j-1}\cup \Delta_{j+1}}$ denotes the sup norm on the indicated set. It is well known that
    $(\mathcal{C}({\vec{\Delta}}),\|\cdot\|_{\vec{\Delta}})$ is a Banach space.

    Let $\mathcal{C}^+({\vec{\Delta}})$ denote the set of all vector functions $\vec{f} = (f_1,\ldots,f_m)$ such that $f_j$ is positive and continuous on $\Delta_{j-1} \cup \Delta_{j+1}$.  On $\mathcal{C}^+({\vec{\Delta}})$ we define the distance
    \[ d(\vec{f},\vec{g}) =\max_{1\leq j\leq m} \left\{ \| \ln(f_j/g_j) \|_{\Delta_{j-1}\cup \Delta_{j+1}}\right\}.
    \]
    Since the logarithm establishes a homeomorphism between $\mathcal{C}^+({\vec{\Delta}})$ and $\mathcal{C}({\vec{\Delta}})$, it is easy to verify that
    $(\mathcal{C}^+({\vec{\Delta}}),d)$ is a complete metric space. Furthermore, we have an important relation between the distance $d(\cdot,\cdot)$  and the norm $\|\cdot\|_{\vec{\Delta}}$ both acting on $\mathcal{C}^+({\vec{\Delta}})$; namely, for any sequence of vector functions $(\vec{g}_n)_{n \geq 0} \subset  \mathcal{C}^+({\vec{\Delta}})$  and $\vec{g} \in \mathcal{C}^+({\vec{\Delta}})$, we have
    \begin{equation}
    	\label{equiv}
    	\lim_{n \to \infty} \|\vec{g}_{n}-\vec{g}\|_{\vec{\Delta}} = 0 \quad  \Leftrightarrow \quad \lim_{n \to \infty} d(\vec{g}_{n},\vec{g}) =0.
    \end{equation}

    \begin{definition}
        \label{df:op:Tw}
        Let $\vec{w} = (w_1,\ldots,w_m)$ be a vector of non-negative functions such that $\ln w_j \in L_1(\eta_{\Delta_j})$, $j=1,\ldots,m$. Define
    \[  \vec{T}_{\vec{w}} = (T_1,\ldots,T_m): \mathcal{C}^+({\vec{\Delta}}) \longrightarrow \mathcal{C}^+({\vec{\Delta}}),
    \]
    where for each $\vec{f} = (f_1,\ldots,f_m)$, $T_j \vec{f}$ is the Szeg\H{o} function on $\overline{\CC} \setminus \Delta_j$ whose boundary values on $\Delta_j$ verify
    \begin{equation}
        \label{boundary}
        |T_j\vec{f}(x)|^2 = \frac{f_{j-1}(x) f_{j+1}(x)}{w_j(x)}, \qquad \mbox{a.e. on}\,\, \Delta_j, \qquad j=1,\ldots,m
    \end{equation}
    ($f_0  \equiv f_{m+1} \equiv 1$).
    \end{definition}

    Recall that Szeg\H{o} functions are positive on the real line outside of the support of the defining measure. Let ${\overline{m}} := \lceil m/2 \rceil$ be the smallest integer greater or equal to $m/2$. As usual,  $\vec{T}_{\vec{w}}^j$  denotes the composition of $\vec{T}_{\vec{w}}$ with itself $j$ times.

    \begin{theorem}
        \label{puntofijo}
        The map $\vec{T}_{\vec{w}}^{\overline{m}}$ is a contraction on the complete metric space $( {C}^+({\vec{\Delta}}),d)$; more precisely,
        \begin{equation}
            \label{contraction}
            d(\vec{T}_{\vec{w}}^{\overline{m}}\vec{f},\vec{T}_{\vec{w}}^{\overline{m}}\vec{g}) \leq  \gamma_m\, d( \vec{f}, \vec{g})
        \end{equation}
        where
        \[ \gamma_m :=
        \left\{
        \begin{array}{cc}
        (2^{\overline{m}} -1)/2^{\overline{m}}, & m \,\,\mathrm{even}, \\
        (2^{\overline{m}} -2)/2^{\overline{m}}, & m \,\,\mathrm{odd}.
        \end{array}
        \right.
        \]
        The map $\vec{T}_{\vec{w}}$ has a unique fixed point.
    \end{theorem}

In \cite[Proposition 1.1]{sasha99} the author proves that
$d(\vec{T}_{\vec{w}} \vec{f},\vec{T}_{\vec{w}} \vec{g}) <
 d( \vec{f}, \vec{g})$ showing that $\vec{T}_{\vec{w}}$ is non expansive and draws the erroneous conclusion that
 $\vec{T}_{\vec{w}}$ is contractive. Theorem \ref{puntofijo} corrects this last statement.

\subsection{Outline of the paper} In Subsections 2.1 and 2.2, we present some properties of the forms $\mathcal{A}_{\n,j}$ and of the polynomials $Q_{\n,j}$ which carry their zeros.  In particular, it is shown that for each $\n \in (\ZZ_+^m)^*$ the polynomials $Q_{\n,j}$ are orthogonal with respect to certain varying measures (see \eqref{ortvar2}). Notice that in the orthogonality relations the polynomials $Q_{\n,j}$ appear at the same time as orthogonal and in the denominator of the varying part of the measure (with the subindices shifted). This inconvenience is settled introducing
an operator $\tilde{T}_{\n}$ on a set of vector polynomials, denoted $\mathcal{P}_{\n}$, for which $(Q_{\n,1}, \ldots, Q_{\n,m})$ is a unique fixed point. In Subsection 2.3 we normalize the orthogonal polynomials and the varying measures  allowing us to use some known results of the theory of orthogonal polynomials with respect to varying measures developed in \cite{lagober98,bernardo_lago1,LGLago} to obtain in Subsection 2.5 certain asymptotic formulas for multiple orthogonal polynomials which are capital in the proof of Theorem \ref{th1} for the following reasons. Lemma \ref{prop5} allows to reduce the proof of the strong asymptotics on compact subsets of $\CC \setminus \Delta_j$ to the strong asymptotics on $\Delta_{j-1} \cup \Delta_{j+1}$ (since \eqref{limden}
implies \eqref{limfund*}) and hints the introduction of the operator $\vec{T}_{\vec{w}}$ (with $\vec{w}$ as in \eqref{weight:dir}) whose fixed point determines the limiting functions of the strong limits of Theorem \ref{th1}. In order to show that the fixed point of $\vec{T}_{\vec{w}}$  attracts fixed points of the operators $\tilde{T}_{\n}$ a local fixed point theorem due
to Brouwer is used and Lemma \ref{prop2} is needed to guarantee that the neighborhoods on which Brouwer's Theorem is applied are non empty.  In Subsection 2.4 the solution of an equilibrium problem, which allows to construct the comparison functions, is studied. Theorem \ref{puntofijo} is proved in Subsection 3.1 with assumptions more general than needed in this paper having in mind future applications. Theorem \ref{th1} is obtained in Subsection 3.2. Subsections 3.3 and 3.4  contain applications of the main result  to give precise estimates of the rate of convergence of ML Hermite-Pad\'e approximations and the strong asymptotics of Cauchy biorthogonal polynomials, respectively.

\section{Auxiliary results}

\subsection{Some useful relations} \label{2.1} The polynomials $Q_{\n,j}$ play an important role throughout the paper so for the benefit of the reader we summarize their construction and some of its properties (for details, see \cite[Lemmas 2.1, 2.2, and 3.1]{SLL22}).

Let $(s_{1,1},\ldots,s_{1,m}) =\mathcal{N}(\sigma_1,\ldots,\sigma_m)$ be given. Assume that there exist polynomials with real coefficients $a_0,\ldots, a_m$ and a monic polynomial $Q$ with real coefficients whose roots lie in $\CC \setminus\Delta_1$ such that
\begin{equation*}
    \frac{\mathcal{A}_0(z)}{Q(z)}\in\mathcal{H}(\CC \setminus\Delta_1) \quad \textrm{ and }\quad \frac{\mathcal{A}_0(z)}{Q(z)}=\mathcal{O}\left(\frac{1}{z^N}\right),\quad z\rightarrow \infty,
\end{equation*}
where $\displaystyle \mathcal{A}_0:= a_0 +\sum_{k=1}^m a_k\widehat{s}_{1,k}$ and $N\geq 1$. Let $\displaystyle \mathcal{A}_1:= a_1 + \sum_{k=2}^m a_k\widehat{s}_{2,k}$. Using Cauchy's integral formula on a contour surrounding $\Delta_1$ and Fubini's Theorem it readily follows (see \cite[Lemma 2.1]{SLL22}) that
\begin{equation} \label{formred}
    \frac{\mathcal{A}_0(z)}{Q(z)} = \int \frac{\mathcal{A}_1(x)}{z-x}\frac{\D \sigma_1(x)}{Q(x)},
\end{equation}
and if $N\geq 2$, from Cauchy's theorem, Fubini's theorem, and Cauchy's integral formula, we have
\begin{equation}
    \label{A:orth}
    \int x^\nu \mathcal{A}_1(x)\frac{\D \sigma_1(x)}{Q(x)}=0,\qquad \nu=0,1,\ldots, N-2.
\end{equation}
In particular, $\mathcal{A}_1$ has at least $N-1$ sign changes in $\mathring{\Delta}_1$ or it is identically equal to zero.

Fix $\n \in (\ZZ_+^m)^*$. The forms $\mathcal{A}_{\n,j}$ in Definition \ref{df:HP:ly} are symmetric with
respect to the real line (that is, $\mathcal{A}_{\n,j}(\overline{z}) = \overline{\mathcal{A}_{\n,j}({z})}$); therefore, their roots come in conjugate pairs. Fix $\n \in (\ZZ_+^m)^*$. If $\mathcal{A}_{\n,0}$ has a zero in $\CC \setminus \Delta_1$ take $Q_{\n,0}$ as a monic polynomial of degree $\geq 1$ with real coefficients whose roots are zeros of $\mathcal{A}_{\n,0}$ in $\CC \setminus \Delta_1$; otherwise, take $Q_{\n,0} \equiv 1$. Using \eqref{formred}-\eqref{A:orth} and  \eqref{df:HP:ly:1} for $j=0$, it follows that
\[
  \frac{\mathcal{A}_{\n,0}(z)}{Q_{\n,0}(z)} = \int \frac{\mathcal{A}_{\n,1}(x)}{z-x}\frac{\D \sigma_1(x)}{Q_{\n,0}(x)},
\]
and
\[
  \int x^\nu \mathcal{A}_{\n,1}(x)\frac{\D \sigma_1(x)}{Q_{\n,0}(x)}=0,\qquad \nu=0,1,\ldots, \eta_{\n,1}-1.
\]
In the second relation there is at least one more orthogonality relation if either the order in \eqref{df:HP:ly:1} for $j=0$ is higher than $\eta_{\n,1}$ or $\deg Q_{\n,0} \geq 1$.

Therefore, there exists a monic polynomial with real coefficients $Q_{\n,1}$ of degree $\geq \eta_{\n,1}$ (or strictly greater if either the order in \eqref{df:HP:ly:1} for $j=0$ is higher than $\eta_{\n,1}$ or $\deg Q_{\n,0} \geq 1$) whose roots are zeros of $\mathcal{A}_{\n,1}$ in $\CC \setminus \Delta_2$ such that $\mathcal{A}_{\n,1}/Q_{\n,1}$ is holomorphic in $\CC \setminus \Delta_2$. Using \eqref{formred}-\eqref{A:orth} and \eqref{df:HP:ly:1} for $j=1$, it follows that
\[
  \frac{\mathcal{A}_{\n,1}(z)}{Q_{\n,1}(z)} = \int \frac{\mathcal{A}_{\n,2}(x)}{z-x}\frac{\D \sigma_2(x)}{Q_{\n,1}(x)}
\]
and
\[
  \int x^\nu \mathcal{A}_{\n,2}(x)\frac{\D \sigma_2(x)}{Q_{\n,1}(x)}=0,\qquad \nu=0,1,\ldots, \eta_{\n,2}-1.
\]
There is at least one more orthogonality relation if either the order in \eqref{df:HP:ly:1} for $j=0,1$ is higher than prescribed,  $\deg Q_{\n,0} \geq 1$, or $\deg Q_{\n,1} > \eta_{\n,1}$.

One can repeat the arguments proving that there exist monic polynomials $Q_{\n,j}$, $\deg Q_{\n,j} \geq \eta_{\n,j}$, $j=0,\ldots,m-1$, $(\eta_{\n,0} = 0)$ with real coefficients whose roots lie in $\CC \setminus \Delta_{j+1}$ and are zeros of $\mathcal{A}_{\n,j}$, such that
\begin{equation}
    \label{formrec}
    \frac{\mathcal{A}_{\n,j}(z)}{Q_{\n,j}(z)} = \int \frac{\mathcal{A}_{\n,j+1}(x)}{z-x}\frac{\D \sigma_{j+1}(x)}{Q_{\n,j}(x)},
\end{equation}
and
\begin{equation}
    \label{ortvar}
    \int x^\nu \mathcal{A}_{\n,j+1}(x) \frac{\D \sigma_{j+1}(x)}{Q_{\n,j}(x)}=0,\qquad \nu=0,1,\ldots, \eta_{\n,j+1}-1,
\end{equation}
with an additional orthogonality relation if for some $j$ either the order in \eqref{df:HP:ly:1} is greater than that prescribed or $\deg Q_{\n,j} > \eta_{\n,j}$. Since $\mathcal{A}_{\n,m} = (-1)^m a_{\n,m}$ and $\deg a_{\n,m} \leq |\n| = \eta_{\n,m}$ the orthogonality relations for $j=m-1$ imply that there can be no additional orthogonality relation unless $a_{\n,m} \equiv 0$. But this is not possible because \eqref{df:HP:ly:1} would imply that $a_{\n,j} \equiv 0$, $j=0,\ldots,m$ and the solution of that system of equations would be the trivial one against our assumption. Consequently, all the orders in \eqref{df:HP:ly:1} are exact, $\deg Q_{\n,j} = \eta_{\n,j}$, $j=0,\ldots,m-1$. and $\deg a_{\n,m} = |\n|$. In the sequel, $Q_{\n,m} := a_{\n,m}$.

 Set
 \[
   \mathcal{H}_{\n,j}(z) := \frac{Q_{\n,j+1}(z) \mathcal{A}_{\n,j}(z)}{Q_{\n,j}(z)}, \quad j=0,\ldots,m-1, \qquad \mathcal{H}_{\n,m}(z) \equiv (-1)^m.
 \]
From \eqref{formrec}-\eqref{ortvar}, for $j=0,\ldots,m-1$, we get
\begin{equation}
    \label{formrec2}
    \mathcal{H}_{\n,j}(z)  = \int \frac{Q_{\n,j+1}^2(x)}{z-x}\frac{\mathcal{H}_{\n,j+1}(x) \D \sigma_{j+1}(x)}{Q_{\n,j}(x)Q_{\n,j+2}(x)},
\end{equation}
 where $Q_{\n,0} \equiv Q_{\n,m+1}  \equiv 1$, and
\begin{equation}
    \label{ortvar2}
    \int x^\nu Q_{\n,j+1}(x) \frac{\mathcal{H}_{\n,j+1}(x) \D \sigma_{j+1}(x)}{Q_{\n,j}(x)Q_{\n,j+2}(x)}=0,\qquad \nu=0,1,\ldots, \eta_{\n,j+1}-1.
 \end{equation}

 Define $K_{\n,m}:= 1$,
    \[
      K_{\n,j}:= \left(\int (Q_{\n,j+1}(x))^2 \frac{| \mathcal{H}_{\n,j+1}(x)\D \sigma_{j+1}(x)|}{| {Q}_{\n,j}(x) {Q}_{\n,j+2}(x)|}\right)^{-1/2},\quad j=0,\ldots,m-1
    \]
    and
    \begin{equation}
        \label{def:kappa}
        \kappa_{\n,j+1}:= \frac{K_{\n,j}}{K_{\n,j+1}},\qquad j=0,\ldots,m-1.
    \end{equation}
    Set
    \begin{equation*}
        \label{orthonormal}
        q_{\n,j} := \kappa_{\n,j}Q_{\n,j}, \quad j=1,\ldots,m, \qquad  {h}_{\n,j} := (K_{n,j})^2 \mathcal{H}_{\n,j} ,\quad j=0,\ldots,m.
    \end{equation*}
 With the new notation, \eqref{formrec2}-\eqref{ortvar2} become
 \begin{equation}
 	\label{formrec3}
    h_{\n,j}(z)  = \varepsilon_{\n,j+1} \int \frac{q_{\n,j+1}^2(x)}{z-x} \frac{|h_{\n,j+1}(x)\D \sigma_{j+1}(x)|}{|Q_{\n,j}(x)Q_{\n,j+2}(x)|},  \qquad j=0,\ldots,m-1
 \end{equation}
 where $\varepsilon_{\n,j+1} $ denotes the constant sign which the varying measure $\frac{\mathcal{H}_{\n,j+1}(x) \D \sigma_{j+1}(x)}{Q_{\n,j}(x)Q_{\n,j+2}(x)}$ adopts on $\Delta_{j+1}$, and
\[
    \int x^\nu q_{\n,j+1}(x) \frac{|h_{n,j+1}(x)\D \sigma_{j+1}(x)|}{|Q_{\n,j}(x)Q_{\n,j+2}(x)|}=0,\qquad \nu=0,1,\ldots, \eta_{\n,j+1}-1.
 \]
 Notice that $q_{\n,j+1}$ is orthonormal with respect to $ \frac{|h_{\n,j+1}(x) \D \sigma_{j+1}(x)|}{|Q_{\n,j}(x)Q_{\n,j+2}(x)|}$.

\subsection{The operator $\tilde{T}_{\n}$}

    Fix $\n \in (\ZZ_+^m)^*$. Let $ \mathcal{P}_{\n} $ denote the set of all vector polynomials $(Q_1,\ldots,Q_m)$ where $Q_j, j=1,\ldots,m,$ is a monic polynomial  with real coefficients of degree $\eta_{\n,j}$
 whose zeros lie in $\mathbb{C} \setminus (\Delta_{j-1} \cup \Delta_{j+1})$. Recall that $\Delta_0 = \Delta_{m+1} \equiv \varnothing$.

    \begin{definition}
        \label{defTn} Define $\tilde{T}_{\n}: \mathcal{P}_{\n} \longrightarrow \mathcal{P}_{\n}$ such that for each ${\vec{Q}} = ({Q}_1,\ldots, {Q}_m) \in \mathcal{P}_{\n}$,  $\tilde{T}_{\n}(\vec{Q}) = \vec{Q}^* = ({Q}^*_1,\ldots,{Q}^*_m)\in \mathcal{P}_{\n}$ satisfies
        \begin{equation}
            \label{defTn1}
            \int x^\nu Q_{j}^*(x)\frac{\mathcal{H}_{j}( {\vec{Q}};x) \D \sigma_{j}(x)}{ {Q}_{j-1}(x) {Q}_{j+1}(x)} =0,\qquad \nu=0,\ldots,\eta_{\n,j} -1, \qquad j=1,\ldots m,
        \end{equation}
        where $\mathcal{H}_{m}({\vec{Q}};x) \equiv (-1)^m$, $Q_0 = Q_{m+1} \equiv 1$, and
        \begin{equation}
            \label{defTn2}
            \mathcal{H}_{j}( {\vec{Q}};z) = \int \frac{(Q^*_{j+1}(x))^2}{z-x}\frac{\mathcal{H}_{j+1}( {\vec{Q}};x)\D\sigma_{j+1}(x)}{ {Q}_{j}(x) {Q}_{j+2}(x)}, \qquad j=0,\ldots,m-1.
        \end{equation}
    \end{definition}

        The components of $\vec{Q}^*$ as well as the weights $\mathcal{H}_{j}( {\vec{Q}};z)$ must be calculated inductively starting from the value $j=m$ down to $j=0$. Since $\mathcal{H}_{m}({\vec{Q}};x)$ is given, with \eqref{defTn1}-\eqref{defTn2} you can determine $Q_m^*$ and $\mathcal{H}_{m-1}({\vec{Q}};x)$ univocally. With this information, \eqref{defTn1}-\eqref{defTn2} determine $Q_{m-1}^*$ and $\mathcal{H}_{m-2}({\vec{Q}};x)$ and so forth. Therefore, $\tilde{T}_{\n}$ is correctly defined.

        We can define a distance $d_{\n}: \mathcal{P}_{\n} \times \mathcal{P}_{\n} \longrightarrow \RR_{\geq 0}$ as follows.
        \[ d_{\n}(\vec{Q}_1, \vec{Q}_2) = \max_{1 \leq j \leq m}\{\|Q_{1,j} - Q_{2,j}\|_{\Delta_{j-1}\cup \Delta_{j+1}}\},
        \]
        where $\vec{Q}_1 = (Q_{1,1},\ldots,Q_{1,m})$ and $\vec{Q}_2 = (Q_{2,1},\ldots,Q_{2,m})$ are arbitrary vector polynomials in $\mathcal{P}_{\n}$ and the norms are uniform on the indicated sets (recall that by convention $\Delta_0 = \Delta_{m+1} = \varnothing$). It is not difficult to prove that $\tilde{T}_{\n}$ is continuous on $\mathcal{P}_{\n}$ with this distance (for details see the corresponding result in \cite[Subsection 3.6]{LGLago}). This property is used when we apply Brouwer's Theorem in Subsection 3.2 below.

    \begin{lemma}
        \label{unicoQn}
        Fix $\n \in (\ZZ_+^m)^*$. The operator $\tilde{T}_{\n}$ has a unique fixed point which coincides with $\vec{Q}_{\n} := (Q_{\n,1},\ldots,Q_{\n,m})$.
    \end{lemma}

    \begin{proof}
    Indeed, $\vec{Q}_{\n} \in \mathcal{P}_{\n}$ and \eqref{formrec2}-\eqref{ortvar2} show that it is a fixed point of $\tilde{T}_{\n}$ (compare with \eqref{defTn1}-\eqref{defTn2}). Let $\vec{\tilde{Q}} = (\tilde{Q}_1,\ldots,\tilde{Q}_m)$ be an arbitrary fixed point of $\tilde{T}_{\n}$. We must show that $\vec{\tilde{Q}}  = \vec{Q}_{\n}$.

    From the definition of the operator $\tilde{T}_{\n}$ and of a fixed point, it follows that
    \begin{equation}
        \label{defTn3}
        \int x^\nu \tilde{Q}_{j} (x)\frac{\mathcal{H}_{j}( {\vec{\tilde{Q}}};x) \D \sigma_{j}(x)}{ \tilde{Q}_{j-1}(x) \tilde{Q}_{j+1}(x)} =0,\qquad \nu=0,\ldots,\eta_{\n,j} -1, \qquad j=1,\ldots m,
    \end{equation}
    where $\mathcal{H}_{m}({\vec{\tilde{Q}}};x) \equiv (-1)^m$, $\tilde{Q}_0 = \tilde{{Q}}_{m+1} \equiv 1$, and
    \begin{equation}
    	\label{defTn4}
        \mathcal{H}_{j}( {\vec{\tilde{Q}}};z) = \int \frac{(\tilde{Q}_{j+1}(x))^2}{z-x}\frac{\mathcal{H}_{j+1}( {\vec{\tilde{Q}}};x)\D\sigma_{j+1}(x)}{ \tilde{Q}_{j}(x) \tilde{Q}_{j+2}(x)}, \qquad j=0,\ldots,m-1.
    \end{equation}
    Notice that $\mathcal{H}_{j}( {\vec{\tilde{Q}}};z)$, $j=0,\ldots,m$ is analytic, symmetric with respect to the real line, and never equals zero in $\CC \setminus \Delta_{j+1}\, (\Delta_{m+1} = \varnothing)$. For $j=m$ this is trivial. Then, from \eqref{defTn4} for $j=m-1$ the assertion readily follows. If for some $j+1$ the statement holds, from \eqref{defTn4} the same is true for $j$.

    Define
    \begin{equation}
        \label{Aj}
        \tilde{\mathcal{A}}_{j}(z) :=  \frac{ \tilde{Q}_{j}(z)\mathcal{H}_{j}( \vec{\tilde{Q}};z)}{\tilde{Q}_{j+1}(z)}, \qquad j=0,\ldots,m.
    \end{equation}
    From this definition and the properties of $\mathcal{H}_{j}( \vec{\tilde{Q}};z)$, it follows that $\tilde{\mathcal{A}}_{j}$, $j=0,\ldots,m$ is analytic and symmetric in $\CC \setminus \Delta_{j+1}$, and its zeros in this region coincide with the roots of $\tilde{Q}_{j}$ all of which are simple and lie on $\Delta_j$. Let us show that there exist polynomials $a_{0}, a_{1},\ldots,a_{m}$, where $\deg a_{j}\leq |\n|-1$, $j=0,1,\ldots,m-1,$ $a_{m}$ monic, and $\deg a_{m}\leq |\n|$, not all identically equal to zero, such that for each $j=0,\ldots,m-1$
    \begin{equation*}
        \tilde{\mathcal{A}}_{j}(z) = \left((-1)^j a_{j} + \sum_{k=j+1}^m (-1)^k a_{k}\widehat{s}_{j+1,k}\right)(z) = \mathcal{O}\left(\frac{1}{z^{n_{j+1}+1}}\right), \qquad z \to \infty.
    \end{equation*}
    If this is so, from the uniqueness of the multi-level Hermite-Pad\'e approximants it follows that
    \[
      (a_{0}, a_{1},\ldots,a_{m}) = (a_{\n,0}, a_{\n,1},\ldots,a_{\n,m}),\qquad (\tilde{\mathcal{A}}_{0}, \tilde{\mathcal{A}}_{1},\ldots,\tilde{\mathcal{A}}_{m-1}) = ( {\mathcal{A}}_{\n,0}, {\mathcal{A}}_{\n,1},\ldots, {\mathcal{A}}_{m-1})
    \]
    and
    \[ (\tilde{Q}_{1},\ldots,\tilde{Q}_{m}) = (Q_{\n,1},\ldots,Q_{\n,m}).
    \]

    Take $a_{m} := \tilde{Q}_m = (-1)^{m}\tilde{\mathcal{A}}_{m}$. Obviously, $a_m$ is a monic  polynomial of degree $|\n|$. With the notation introduced in \eqref{Aj}, \eqref{defTn3}-\eqref{defTn4} adopt the form
    \begin{equation}
        \label{defTn5}
        \int x^\nu \tilde{\mathcal{A}}_{j}(x) \frac{  \D \sigma_{j}(x)}{ \tilde{Q}_{j-1}(x) } =0,\qquad \nu=0,\ldots,\eta_{\n,j} -1, \qquad j=1,\ldots m,
    \end{equation}
    and
    \begin{equation}
        \label{defTn6}
        \tilde{\mathcal{A}}_{j}(z)  = \frac{\tilde{Q}_{j}(z)}{\tilde{Q}_{j+1}(z)}\int \frac{\tilde{Q}_{j+1}(x)}{z-x}\frac{\tilde{\mathcal{A}}_{j+1}(x)\D\sigma_{j+1}(x)}{ \tilde{Q}_{j}(x) }, \qquad j=0,\ldots,m-1.
    \end{equation}
    Since $\deg \tilde{Q}_{j} = \eta_{\n,j}$ and $\deg \tilde{Q}_{j+1} =\eta_{\n,j+1}$ this relation entails that
    \[ \tilde{\mathcal{A}}_{j}(z) = \mathcal{O}(1/z^{n_{j+1} +1}), \qquad z \to \infty,
    \]
    as needed. Using \eqref{defTn5}-\eqref{defTn6}, we get
    \begin{equation}
        \label{defTn7}
        \tilde{\mathcal{A}}_{j}(z) = \int \frac{\tilde{\mathcal{A}}_{j+1}(x) }{z-x} \D\sigma_{j+1}(x).
    \end{equation}
    For $j= m-1$, \eqref{defTn7} gives
    \[
        \tilde{\mathcal{A}}_{m-1}(z) =   \int \frac{ \tilde{\mathcal{A}}_{m}(x)-  \tilde{\mathcal{A}}_{m}(z)}{z-x} \D\sigma_{m}(x) + \int \frac{  \tilde{\mathcal{A}}_{m}(z)}{z-x} \D\sigma_{m}(x)
    \]
    \[ = (-1)^{m-1}a_{m-1}(z)  + (-1)^m a_m(z) \widehat{s}_{m,m},
    \]
    where
    \[
      a_{m-1}(z):= \int \frac{{a}_{m}(z)-{a}_{m}(x)}{z-x}\D\sigma_{m}(x), \qquad \deg a_{m-1} = |\n| -1.
    \]

    Assume that there exist polynomials $a_{j+1},\ldots,a_{m-1}$, $\deg a_k = |\n| -1$, $k= j+1,\ldots,m-1,$ such that
    \[ \tilde{\mathcal{A}}_{j+1} = (-1)^{j+1} a_{j+1} + \sum_{k=j+2}^m (-1)^k a_{k}\widehat{s}_{j+2,k}.
    \]
    Substituting this expression inside the integral in \eqref{defTn7}, we get
    \[
      \tilde{\mathcal{A}}_{j}(z) = \int \frac{\tilde{\mathcal{A}}_{j+1}(x) \mp\left((-1)^{j+1} a_{j+1}(z) + \sum_{k=j+2}^m (-1)^k a_{k}(z)\widehat{s}_{j+2,k}(x)\right)}{z-x} \D\sigma_{j+1}(x) =
    \]
    \[
      (-1)^{j+1}\int \frac{a_{j+1}(x) -a_{j+1}(z)}{z-x} \D\sigma_{j+1}(x) + \sum_{k= j+2}^m (-1)^k \int \frac{a_{k}(x) -a_{k}(z)}{z-x} \D s_{j+1,k}(x) +
    \]
    \[  \sum_{k=j+1}^m (-1)^{k}a_{k}(z)\widehat{s} _{j+1,k}(z).
    \]
    obtaining  for $ \tilde{\mathcal{A}}_{j}$ an expression similar to the one we had for $ \tilde{\mathcal{A}}_{j+1}$ with the same $a_k$, $k=j+1,\ldots,m$  plus
    \[
      a_j(z) :=  \int \frac{a_{j+1}(z) -a_{j+1}(x)}{z-x} \D {s}_{j+1,j+1}(x) + \sum_{k= j+2}^m (-1)^{k-j+1} \int \frac{a_{k}(z) -a_{k}(x)}{z-x} \D s_{j+1,k}(x),
    \]
    which obviously is a polynomial of degree $|\n|-1$. We have concluded the proof.
    \end{proof}

    In the definition of the operator $\tilde{T}_{\n}$ one could be tempted to replace the weights $\mathcal{H}_j(\vec{Q};x)$ with $\mathcal{H}_{\n,j}(x)$ (independent of $\vec{Q}$). The resulting operator also has $\vec{Q}_{\n}$ as a fixed point but it is not clear whether it has other fixed points or not.

\subsection{Normalization}  \label{normalization}

    Fix $\n \in (\ZZ_+^m)^*$. Let $\vec{Q} = (Q_1,\ldots,Q_m) \in \mathcal{P}_{\n}$ and $\vec{Q}^* = (Q_1^*,\ldots,Q_m^*) = \tilde{T}_{\n}(\vec{Q})$. Define
    \[
      K_j(\vec{Q}^*) := \left(\int (Q^*_{j+1}(x))^2 \frac{| \mathcal{H}_{j+1}(\vec{Q};x)\D \sigma_{j+1}(x)|}{| {Q}_{j}(x) {Q}_{j+2}(x)|}\right)^{-1/2},\quad j=0,\ldots,m-1,
    \]
    \[
      K_{m}(\vec{Q}^*):= 1,\qquad
      \kappa_{j+1}(\vec{Q}^*):= \frac{K_{ j}(\vec{Q}^*)}{K_{ j+1}(\vec{Q}^*)},\qquad j=0,\ldots,m-1.
    \]
    Set
    \begin{equation*}
        q_{j}^* := \kappa_{ j}(\vec{Q}^*)Q_{ j}^*, \quad j=1,\ldots,m, \qquad  {h}_{ j}(\vec{Q};x) := (K_{j}(\vec{Q}^*))^2 \mathcal{H}_{j}(\vec{Q};x) ,\quad j=0,\ldots,m.
    \end{equation*}

    With this normalization, \eqref{defTn1}--\eqref{defTn2} imply
    \[
        \int x^\nu q_{j}^*(x)\frac{| {h}_{j}( {\vec{Q}};x)\D \sigma_{j}(x)|}{ |{Q}_{j-1}(x) {Q}_{j+1}(x)|} =0,\quad \nu=0,\ldots,\eta_{\n,j} -1, \quad j=1,\ldots m,
    \]
    where $ {h}_{m}({\vec{Q}};x) \equiv (-1)^m$, $Q_0 = Q_{m+1} \equiv 1$, and
    \begin{equation}
    	\label{defTn12}
        {h}_{j}( {\vec{Q}};z) = \varepsilon_{j+1}(\vec{Q})\int \frac{(q^*_{j+1}(x))^2}{z-x}\frac{| {h}_{j+1}( {\vec{Q}};x)| |\D\sigma_{j+1}(x)|}{| {Q}_{j}(x) {Q}_{j+2}(x)|}, \qquad j=0,\ldots,m-1,
    \end{equation}
    where $\varepsilon_{j+1}(\vec{Q})$ denotes the constant sign which the varying measure $\frac{{h}_{j+1}( {\vec{Q}};x) \D\sigma_{j+1}(x)}{{Q}_{j}(x) {Q}_{j+2}(x)}$ adopts on $\Delta_{j+1}$. Notice that $q^*_{j+1}$ is orthonormal with respect to $ \frac{|{h}_{j+1}( {\vec{Q}};x)\D\sigma_{j+1}(x)|}{|{Q}_{j}(x) {Q}_{j+2}(x)|}$.

    \begin{lemma}
        \label{lem:hnj}
        Suppose that $\sigma_j'\neq 0 $ a.e. on $\Delta_j = [a_j,b_j]$, $j=1,\ldots,m $.  Let $\Lambda \subset (\ZZ_+^m)^*$ be a sequence of distinct multi-indices such that for some fixed $N \in \mathbb{Z}_+$
         \begin{equation}
        \label{N}
        n_{j+1} \leq n_j + N, \qquad j=1,\ldots,m-1, \qquad \n = (n_1,\ldots,n_m) \in \Lambda,
    \end{equation}
          and let $(\vec{\tilde{Q}}_{\n})_{\n \in \Lambda}, \vec{\tilde{Q}}_{\n} \in \mathcal{P}_{\n},$ be an arbitrary sequence of vector polynomials such that  for each $j=1,\ldots,m$ the zeros of the $j-th$ component $\tilde{Q}_{\n,j}$ of $\vec{\tilde{Q}}_{\n}$ remain uniformly bounded away from $\Delta_{j-1} \cup \Delta_{j+1}$ for $\n \in \Lambda$.    Then
        \begin{equation}
            \label{def:htil}
            \lim_{\n \in \Lambda} |{h}_{j}( {\vec{\tilde{Q}}}_{\n};z)| = \frac{1}{\sqrt{|z- b_{j+1}||z-a_{j+1}|}} =: \tilde{h}_j(z), \qquad j=0,\ldots,m-1,
        \end{equation}
        uniformly on $\Delta_{j} \cup \Delta_{j+2}$, $(\Delta_0 = \Delta_{m+1} = \varnothing)$.
    \end{lemma}
    \begin{proof}
    In the sequel, $\vec{{Q}}_{\n}^* = \tilde{T}_{\n}(\vec{\tilde{Q}}_{\n} )$, where $\vec{{Q}}_{\n}^* = ({Q}_{\n,1}^* ,\ldots,{Q}_{\n,m}^* )$, $\vec{\tilde{Q}}_{\n} = (\tilde{Q}_{\n,1},\ldots,\tilde{Q}_{\n,m})$, and $({q}_{\n,1}^* ,\ldots,{q}_{\n,m}^* )$ is the vector of orthonormal polynomials associated with $({Q}_{\n,1}^* ,\ldots,{Q}_{\n,m}^* )$ using the normalization introduced previously. According to \eqref{defTn12}
    \begin{equation}
        \label{defTn13}
        {h}_{j}( {\vec{\tilde{Q}}}_{\n};z) = \varepsilon_{j+1}(\vec{\tilde{Q}}_{\n})\int \frac{(q^*_{\n,j+1}(x))^2}{z-x}\frac{| {h}_{j+1}( {\vec{\tilde{Q}}}_{\n};x)\D\sigma_{j+1}(x)|}{|\tilde{Q}_{\n,j}(x) \tilde{Q}_{\n,j+2}(x)|}, \qquad j=0,\ldots,m-1.
    \end{equation}
    As usual, $\tilde{Q}_{\n,0} = \tilde{Q}_{\n,m+1} \equiv 1$.

   By assumption, the zeros of the polynomials $\tilde{Q}_{\n,j} \tilde{Q}_{\n,j+2}$ are uniformly bounded away from $\Delta_{j+1}$. Due to \eqref{N}, we have
    \[ \deg\tilde{Q}_{\n,j}(x) + \deg \tilde{Q}_{\n,j+2}(x)-2\deg q^*_{\n,j+1} = n_{j+2}-n_{j+1} \leq N
    \]
    for all $\n \in \Lambda$,  $q^*_{\n,j+1}(x)$ is orthonormal with respect to the varying measure $\frac{| {h}_{j+1}( {\vec{\tilde{Q}}}_n;x)\D\sigma_{j+1}(x)|}{|\tilde{Q}_{\n,j}(x) \tilde{Q}_{\n,j+2}(x)|}$, and $\sigma_{j+1}' > 0$ a.e. on $\Delta_{j+1}$. We have all the ingredients to use \cite[Theorem 8]{bernardo_lago1} (see also \cite{lagober98}) and we obtain
    \begin{multline*}
        \lim_{\n \in \Lambda} \int \frac{(q^*_{\n,j+1}(x))^2}{z-x}\frac{| {h}_{j+1}( {\vec{\tilde{Q}}}_{\n};x)\D\sigma_{j+1}(x)|}{|\tilde{Q}_{\n,j}(x) \tilde{Q}_{\n,j+2}(x)|} = \\
        \frac{1}{\pi} \int_{a_{j+1}}^{b_{j+1}} \frac{1}{z-x} \frac{\D x}{\sqrt{(b_{j+1}-x)(x-a_{j+1})}} = \frac{1}{\sqrt{(z-b_{j+1})(z - a_{j+1})}},
    \end{multline*}
    uniformly on compact subsets of $\overline{\CC} \setminus \Delta_{j+1}$, provided   $|{h}_{j+1}( {\vec{\tilde{Q}}}_{\n};x)|$ converges uniformly to a positive function on $\Delta_{j+1}$. This can be proved by induction.

When $j=m-1$, $|{h}_{m}( {\vec{\tilde{Q}}}_{\n};x)| \equiv 1$ and the property holds; therefore,
\[
  \lim_{\n \in \Lambda} |{h}_{m-1}( {\vec{\tilde{Q}}}_{\n};z)| =\lim_{\n \in \Lambda}  \left|\int \frac{(q^*_{\n,m}(x))^2}{z-x}\frac{| {h}_{m}( {\vec{\tilde{Q}}}_{\n};x) \D\sigma_{m}(x)|}{|\tilde{Q}_{\n,m-1}(x)|~}\right| =  \frac{1}{|\sqrt{(z-a_m)(z-b_m)}|} ,
\]
uniformly on compact subsets of $\overline{\CC} \setminus \Delta_m$; in particular, on $\Delta_{m-1}$. Now, if we assume that the property holds for some $j+1$, repeating the arguments and using \eqref{defTn13}, we obtain that it also holds for $j$.
\end{proof}

\subsection{Some properties of the vector equilibrium measure}

The equilibrium measure which is relevant in our study satisfies some additional properties.

\begin{lemma}
    \label{soportecompleto}
    Let $ (p_1,\ldots,p_m) \in (\ZZ_+^m)^*$ where $p_1 > 0$ and $p_1 \geq \cdots \geq p_m.$  Let $\vec{\lambda} \in \mathcal{M}_1(\vec{\Delta})$ be the unitary vector equilibrium measure which solves \eqref{eq:vector}  with interaction matrix \eqref{interaction}  over the system of intervals $\vec{\Delta}$ and $P_j, j=1,\ldots,m,$ is defined by \eqref{Ps}. Then, for each $j=1,\ldots, m$,
    \begin{equation}
    	\label{soporte}
    	\supp (\lambda_j) = \Delta_j = [a_j,b_j],
    \end{equation}
    where $\vec{\lambda}= (\lambda_1,\ldots,\lambda_m)$. Moreover, $\D \lambda_j(x) = \lambda_j'(x) \D x$, $j=1,\ldots,m,$ is absolutely continuous with respect to  Lebesque measure on $\Delta_j$ and there exist constants  $A> 0, \beta =-1/2,$ and $\beta_0$ such that
	\begin{equation}
		\label{cond:a*}
		A^{-1}((b_j-x)(x-a_j))^{\beta_0}\leq \lambda_j'(x) \leq A((b_j-x)(x-a_j))^\beta, \qquad x \in (a_j,b_j).
	\end{equation}
\end{lemma}

\begin{proof}
    Fix $j$, $1\leq j \leq m$. We can rewrite \eqref{eq:vector} in the form
    \begin{equation}
        \label{extfield}
        V^{\lambda_j}(x)  - \frac{1}{2} V^{\tau_j}(x)
        \begin{cases}
            =\omega_j, & x\in\supp \lambda_j,\\
            \geq \omega_j, & x\in \Delta_j,
        \end{cases}
    \end{equation}
where
\[ \tau_j = \frac{P_{j-1}}{P_j} \lambda_{j-1}+ \frac{P_{j+1}}{P_j} \lambda_{j+1}
\]
is a measure whose support is contained in $\Delta_{j-1} \cup \Delta_{j+1} $ which is a set disjoint from $\Delta_j$. The total mass of $\tau_j$ is
\[ \frac{P_{j-1}}{P_j} + \frac{P_{j+1}}{P_j} =: c_j.
\]
Obviously $c_j \leq 2$ since by hypothesis $P_{j-1} + P_{j+1} - 2P_j = p_{j+1} - p_j \leq 0$. This fact together with \eqref{extfield} imply that $2\lambda_j$ is the balayage of $\tau_j$ onto $\Delta_j$ plus $(2-c_j)$ times the equilibrium measure of the interval $\Delta_j$ (see  \cite[Ex. 4.8, pg.118]{SaffT}). The support of such a measure verifies \eqref{soporte} (see  \cite[Cor. 4.12, pg.122]{SaffT}).

On the other hand, the external field $-V^{\tau_j}/2$ acting on $\Delta_j$  (see \eqref{extfield}) is infinitely  differentiable on that interval. Using \cite[Theorem 1.34]{DKL} and \cite[Lemma 4]{KD} the analytic
properties  of $\lambda_j$ readily follow.\end{proof}

\subsection{Auxiliary asymptotic relations}
The operator $\vec{T}_{\n}$ has the remarkable property that it amplifies the region where strong asymptotics takes place in the sense expressed by the following lemma.

Everywhere below, if $w \D \mu $ is in the Szeg\H{o} class, where $w$ is a weight and $\mu$ a positive measure, then
$\mathsf{G}(w\mu,z)$ denotes the Szeg\H{o} function with respect to the measure $w \D \mu$.

    \begin{lemma}
        \label{prop5}
        Let $\Lambda = \Lambda(p_1,\ldots,p_m) \subset (\ZZ_+^m)^*$ be a ray sequence with  $p_1 \geq \cdots \geq p_m,$ and \eqref{const_ray} takes place. Let $\vec{\lambda} \in \mathcal{M}_1(\vec{\Delta})$ be the unitary vector equilibrium measure which solves \eqref{eq:vector}  with interaction matrix \eqref{interaction} over the system of intervals $\vec{\Delta}$ and $P_j, j=1,\ldots,m,$ is defined by \eqref{Ps}. Take $\Phi_j$ and $C_j, j=1,\ldots,m$ as in \eqref{compar}. Let $(\vec{\tilde{Q}}_{\n})_{\n\in\Lambda}, \vec{\tilde{Q}}_{\n} \in \mathcal{P}_{\n},$ be a sequence of vector polynomials such that  for each $j=1,\ldots,m$ the zeros of the $j-th$ component $\tilde{Q}_{\n,j}$ of $\vec{\tilde{Q}}_{\n}$ remain uniformly bounded away from $\Delta_{j-1} \cup \Delta_{j+1}$ for $\n \in \Lambda$.   Set
        \[ \vec{f}_{\n} = \left(\frac{\tilde{Q}_{\n,1}}{\Phi_1^{\eta_{\n,1}}}, \cdots, \frac{\tilde{Q}_{\n,m}}{\Phi_m^{\eta_{\n,m}}}\right).
        \]
        Assume that there exists $\vec{f}=(f_1,\ldots,f_m)\in\mathcal{C}^+(\vec{\Delta})$  such that
        \begin{equation}
            \label{limden}
            \lim_{\n\in\Lambda} \|\vec{f}_{\n}-\vec{f}\|_{\vec{\Delta}} =0.
        \end{equation}
        Let $\vec{Q}_{\n}^* = \tilde{T}_{\n}(\vec{\tilde{Q}}_{\n}) = (Q_{\n,1}^*,\ldots,Q_{\n,j}^*)$ and   ${q}_{\n,j}^* =
        \kappa_j(\vec{Q}_{\n}^*) {Q}_{\n,j}^*$, $j=1,\ldots,m$. Suppose that  $\vec{\sigma}\in\mathcal{S}(\vec{\Delta})$. Then
        \[
            \lim_{\n\in\Lambda} \frac{q^*_{\n,j}(z)}{[C_j\Phi_j(z)]^{\eta_{\n,j}}} = \frac{1}{\sqrt{2\pi}}\mathsf{G}\left((f_{j-1}f_{j+1})^{-1}\tilde{h}_j\sigma_j,z\right), \qquad  j=1,\ldots,m,
        \]
        uniformly on compact subsets of $\overline{\CC}\setminus \Delta_j$ ($f_0\equiv f_{m+1}\equiv 1$) where $\tilde{h}_j$, $j=1,\ldots,m-1$ is given in \eqref{def:htil} and $\tilde{h}_m \equiv 1$. In addition,
        \begin{equation*}
            \lim_{\n\in\Lambda} \frac{\kappa_{j}(\vec{Q}_{\n}^*)}{C_j^{\eta_{\n,j}}} = \frac{1}{\sqrt{2\pi}}\mathsf{G}\left((f_{j-1}f_{j+1})^{-1}\tilde{h}_j\sigma_j,\infty\right),\qquad j=1,\ldots,m.
        \end{equation*}
                Consequently,
        \begin{equation}
            \label{limfund*}
            \lim_{\n\in\Lambda} \frac{Q^*_{\n,j}(z)}{\Phi_j^{\eta_{\n,j}}(z)} = \frac{\mathsf{G}\left((f_{j-1}f_{j+1})^{-1}\tilde{h}_j\sigma_j,z\right)}{\mathsf{G}\left((f_{j-1}f_{j+1})^{-1}\tilde{h}_j\sigma_j,\infty\right)}, \qquad j=1,\ldots,m,
        \end{equation}
     uniformly on compact subsets of $\overline{\CC}\setminus \Delta_j$.
    \end{lemma}
    \begin{proof} Due to the assumptions, Lemma \ref{soportecompleto} guarantees that $\supp \lambda_j = \Delta_j$, $j=1,\ldots,m$ and we have equality in \eqref{eq:vector} on all $\Delta_j $ for each $j=1,\ldots,m$.

     To obtain the limit relations we apply \cite[Th. 1.2]{LGLago} which is a theorem on the strong asymptotics of orthogonal polynomials with respect to varying measures. In that theorem, $n$ is a parameter which runs through all the natural numbers. That result remains valid if $n$ runs through any infinite sequence of distinct natural numbers. For our purpose what is important is that $n$ controls the degree of the orthogonal polynomial under consideration.	

    Fix $j  \in \{1,\ldots,m\}$ and consider the sequence of varying measures
    \[ (| { {h}_{j}( {\vec{\tilde{Q}}_{\n}};x)| \D\sigma_{j}(x)/ |\tilde{Q}_{\n,j-1}(x) \tilde{Q}_{\n,j+1}(x)}|), \qquad \n \in \Lambda,
    \]
    where $\tilde{Q}_{\n,0} = \tilde{Q}_{\n,m+1} \equiv 1$ and $| {h}_{m}( {\vec{\tilde{Q}}_{\n}};x)| \equiv 1 $.
    We have to show that this sequence of varying measures satisfies all the assumptions i)-iv) of \cite[Th. 1.2]{LGLago} taking
   \[ \D \mu_n(x) = | {h}_{j}( {\vec{\tilde{Q}}_{\n}};x)| \D\sigma_{j}(x), \qquad w_{2n}(x) = \tilde{Q}_{\n,j-1}(x) \tilde{Q}_{\n,j+1}(x), \qquad n = \deg Q_{\n,j}^* = \eta_{\n,j}.
   \]
   where $\mu_n$ and $ w_{2n}$ follow the notation in \cite{LGLago}.

   Because of Lemma \ref{lem:hnj}, $\lim_{\n \in \Lambda} | {h}_{j}( {\vec{\tilde{Q}}_{\n}};x)| = \tilde{h}_j(x)$ uniformly on $\Delta_j$ and $ \sigma_j' \neq 0 $ a.e. on $\Delta_j$. Consequently, i) readily follows with $\D \mu = \tilde{h}_j  \D \sigma_j$. Since $\sigma_j \in \mathcal{S}(\Delta_j)$ and $\tilde{h}_j(x) > 0, x \in \Delta_j$, ii) is also immediate.  The components of $\n$ are decreasing, so
    \[ \deg w_{2n} = \deg(\tilde{Q}_{\n,j-1} \tilde{Q}_{\n,j+1}) = \eta_{\n,j-1} + \eta_{\n,j+1} = 2 \deg Q_{\n,j} + n_{j+1} - n_j \leq 2 \deg Q_{\n,j} = 2n
    \]
    and the zeros of $w_{2n}$ are bounded away from $\Delta_{j}$ for $\n \in \Lambda$; therefore, iii) holds. Finally,  taking into account  \eqref{compar} and \eqref{const_ray}, we get
    \[  \left|\Phi_{j-1}^{\eta_{\n,j-1}}(x)\Phi_{j+1}^{\eta_{\n,j+1}}(x)\right|^{-1} = \left|\Phi_{j-1}^{P_{j-1}/P_j}(x)\Phi_{j+1}^{P_{j+1}/P_j}(x)\right|^{-\eta_{\n,j}} =
    \]
    \[\exp\left[\eta_{\n,j}\left( \frac{P_{j-1}}{P_j} V^{\lambda_{j-1}}(x) +  \frac{P_{j+1}}{P_j} V^{\lambda_{j+1}}(x)\right)\right], \qquad x \in \Delta_j.
    \]
    Set
    \[ \varphi_j (x) := \exp \left( \frac{P_{j-1}}{P_j} V^{\lambda_{j-1}}(x) +  \frac{P_{j+1}}{P_j} V^{\lambda_{j+1}}(x)\right).
    \]
    Using these identities and \eqref{limden}, we obtain
    \[ \lim_{\n \in \Lambda}  \varphi_j^{\eta_{\n,j}} (x)  |w_{2n}(x)| =
    \]
   \[ \lim_{\n \in \Lambda} \frac{|\tilde{Q}_{\n,j-1}(x)|}{|\Phi_{j-1}^{\eta_{\n,j-1}}(x)|}\frac{ | \tilde{Q}_{\n,j+1}(x)|}{|\Phi_{j+1}^{\eta_{\n,j+1}}(x)|} = f_{j-1}(x) f_{j+1}(x) =: 1/\psi_j(x),
   \]
   uniformly on $\Delta_j$,  and iv) takes place with $\varphi = \varphi_j$ and $\psi = \psi_j$.
Thus, the external field of the equilibrium problem in   \cite[(1.5)]{LGLago} which allows to describe the limit relations in \cite[Th. 1.2]{LGLago} is
\[ -\frac{1}{2} \ln \varphi_j(x) =  - \frac{1}{2}\left( \frac{P_{j-1}}{P_j}V^{\lambda_{j-1}(x)} + \frac{P_{j+1}}{P_j}V^{\lambda_{j+1}(x)} \right)
\]
which coincides with the one in \eqref{extfield}. The equilibrium measure and equilibrium constant are $\lambda_j$ and $w_j$, respectively. Now, the conclusions of Lemma \ref{prop5} follow directly from \cite[Th. 1.2]{LGLago}.
    \end{proof}

A clear advantage of this Lemma is that in order to find the strong asymptotics for sequences of fixed points of the operators $\tilde{T}_{\n}$ it is sufficient to restrict the analysis to the intervals $\Delta_j$ to obtain \eqref{limden}. This result also indicates that if the sequence of fixed points of the operators $\tilde{T}_{\n}$ satisfy \eqref{limden}, then
$\vec{f}$ has to be a fixed point of the operator $\vec{T}_{\vec{w}}$ where $\vec{w}$ is given by \eqref{weight:dir}.

The analogue of Lemma \ref{prop5} in \cite{sasha99} is Proposition 1.2. For its proof, Aptekarev uses a different technique close to Widom's approach in \cite{widom} for the study of the strong asymptotics of orthogonal polynomials on arcs and curves of the complex plane.  It has the advantage that one obtains $L_2$ estimates on the intervals $\Delta_j$ of the strong asymptotic formula, but it has the drawback that one must restrict the analysis to Nikishin systems generated
by weights  (instead of measures) with Szeg\H{o}'s condition. We could have followed that approach but prefered to maintain the generality and the use of the theory of orthogonal polynomials with respect to varying measures.

A natural question is if there are sequences of vector polynomials for which \eqref{limden} holds.  As the next lemma shows, such sequences of vector polynomials may be given for a large class of  predetermined limiting vector functions $\vec{f}$. This is relevant in the proof of Theorem \ref{th1}.

\begin{lemma}
    \label{prop2}
    Assume that $\vec{\mu} = (\mu_1,\ldots,\mu_m)\in\mathcal{S}(\vec{\Delta})$, and for each $\n\in (\ZZ^m)^* $, $(\tilde{q}_{\n,1}, \ldots, \tilde{q}_{\n,m})$ denotes the vector polynomial whose components have positive leading coefficients, verify $\deg \tilde{q}_{\n,j} = \eta_{\n,j}$, $j=1,\ldots,m,$ and
    \begin{equation}
        \label{poltilde}
        \int x^\nu \tilde{q}_{\n,j}(x)\frac{\D \mu_j(x)}{|\Phi_j^{2\eta_{\n,j}}(x)|} = 0,\qquad 0\leq\nu < \eta_{\n,j},\qquad  \int  [\tilde{q}_{\n,j}(x)]^2\frac{\D \mu_j(x)}{|\Phi_j^{2\eta_{\n,j}}(x)|} = 1.
    \end{equation}
    Then, for each ray sequence $\Lambda = \Lambda(p_1,\ldots,p_m)$ such that  \eqref{N} takes place,
    \begin{equation}
        \label{predet}
        \lim_{\n\in\Lambda} \frac{\tilde{q}_{\n,j}(z)}{\Phi_j^{\eta_{\n,j}}(z)} = \frac{\mathsf{G}(\mu_j,z)}{\sqrt{2\pi}},
    \end{equation}
    and
    \begin{equation}
        \label{predet*}
        \lim_{\n\in\Lambda} \frac{\tilde{Q}_{\n,j}(z)}{\Phi_j^{\eta_{\n,j}}(z)} = \frac{\mathsf{G}(\mu_j,z)}{\mathsf{G}(\mu_j,\infty)},
    \end{equation}
    uniformly on each compact subset of $\overline{\CC} \setminus \Delta_j, j=1,\ldots,m,$ where the $\Phi_j$ were introduced in \eqref{compar} from the equilibrium problem corresponding to $(p_1,\ldots,p_m)$ and $\tilde{Q}_{\n,j}$, $j=1,\ldots,m$ is $\tilde{q}_{\n,j}$ renormalized to be monic.
\end{lemma}

\begin{proof} Let $\Lambda$ be a ray sequence with the indicated restrictions. In particular, \eqref{N} implies that
$p_1\geq \cdots \geq p_m$ with $p_1 > 0$. Then, the assumptions of Lemma \ref{soportecompleto} are satisfied and, consequently, \eqref{soporte} and \eqref{cond:a*} take place. This means that for each $j=1,\ldots,m$ the hypotheses  of \cite[Theorem 1.3]{LGLago} hold true and that result guarantees that the sequences of polynomials  $(\tilde{q}_{\n,j})_{\n \in \Lambda}$ defined by \eqref{poltilde} fulfill \eqref{predet} and \eqref{predet*}. \end{proof}

In \cite{sasha99}, Theorem 2 is the analogue of Lemma \ref{prop2}. The proof of  \cite[Theorem 2]{sasha99} is based on an intricate theoretical construction using meromorphic functions on a Riemann surface which in our opinion has its weak points. Here Aptekarev also uses that the generating measures of the Nikishin system are absolutely continuous with respect to Lebesgue measure. Our proof gives an explicit solution in terms of orthogonal polynomials, the arguments are quite simple, and do not require absolute continuity of the generating measures.

\section{Proof of main results}

\subsection{Proof of Theorem \ref{puntofijo}}
For simplicity, in the proof we write $\vec{T}_{\vec{w}} = \vec{T}$. Let us assume that \eqref{contraction} takes place. It is known that if such an $\overline{m}$ exists then  $\vec{T}$ has a unique fixed point (see, for example,  \cite[Ex. (A.1), pg. 17]{GG}). The proof is simple and for completeness we include it.

 Since $( {C}^+ (\vec{\Delta}),d)$ is a complete metric space and $\gamma_m < 1$, by Banach's fixed point theorem    $\vec{T}^{\overline{m}}$ has a unique fixed point, say $\vec{f}^*$. In particular, $\vec{T}^{\overline{m}}\vec{f}^* = \vec{f}^*$. Applying $\vec{T}$ one more time, we get
 \[ \vec{T}^{\overline{m}+1}\vec{f}^* = \vec{T}\vec{f}^* = \vec{T}^{\overline{m} } \vec{T}\vec{f}^*.
 \]
 Consequently, $\vec{T}\vec{f}^*$ is also a fixed point of $\vec{T}^{\overline{m}}$. But the fixed point of $\vec{T}^{\overline{m}}$ is unique, so $\vec{T}\vec{f}^* = \vec{f}^*$; therefore, $\vec{f}^*$ is a fixed point of $\vec{T}$. On the other hand, if $\vec{f}$ is any fixed point of $\vec{T}$ we have
 \[ \vec{T}^{\overline{m}}\vec{f} = \vec{T}^{\overline{m}-1}\vec{f} = \cdots = \vec{T}\vec{f} = \vec{f}.
 \]
 So, $ \vec{f}$ is also a fixed point of $\vec{T}^{\overline{m}}$. Consequently, $ \vec{T}$ cannot have two distinct fixed points since they would be distinct fixed points of $\vec{T}^{\overline{m}}$.

 Let us prove \eqref{contraction}. For general $m$ the notation becomes a bit cumbersome so to fix ideas we begin with the simple case $m=2$. Thus $\overline{m}=1$ and we must show that $\vec{T}$ is contractive with $\gamma_2 = 1/2$. Take arbitrary functions $\vec{f} = (f_1,f_2), \vec{g} = (g_1,g_2) \in  {C}^+({\vec{\Delta}})$, and $\vec{\Delta} = (\Delta_1,\Delta_2)$. We have
 \[
   d(\vec{T}\vec{f},\vec{T}\vec{g}) :=  \max\left\{  \|\ln (|T_1\vec{f}|/|T_1\vec{g}|)\|_{\Delta_2},   \|\ln (|T_2\vec{f}|/|T_2\vec{g}|)\|_{\Delta_1} \right\}.
  \]
 Now, $\ln (|T_1\vec{f}|/|T_1\vec{g}|)$ and $ \ln (|T_2\vec{f}|/|T_2\vec{g}|)$ are harmonic functions on $\overline{\CC} \setminus \Delta_1$ and $\overline{\CC} \setminus \Delta_2$, respectively. Using the maximum principle for harmonic functions it follows that
 \begin{multline*}
 	\max\left\{  \|\ln (|T_1\vec{f}|/|T_1\vec{g}|)\|_{\Delta_2},   \|\ln (|T_2\vec{f}|/|T_2\vec{g}|)\|_{\Delta_1} \right\} \\ \leq \max\left\{\sup_{\Delta_1} |\ln (|T_1\vec{f}|/|T_1\vec{g}|)| , \sup_{\Delta_2} |\ln (|T_2\vec{f}|/|T_2\vec{g}|)|\right\}.
 \end{multline*}
  Due to \eqref{boundary}
  \begin{multline*}
  	\max\left\{\sup_{\Delta_1} |\ln (|T_1\vec{f}|/|T_1\vec{g}|)| , \sup_{\Delta_2} |\ln (|T_2\vec{f}|/|T_2\vec{g}|)|\right\} \\
  	= \frac{1}{2} \max\left\{\|\ln (f_2/g_2)\|_{\Delta_1} ,\|\ln (f_1/ g_1)\|_{\Delta_2}\right\} = \frac{1}{2} d(\vec{f},\vec{g}).
  \end{multline*}

  Notice that in the quotient the weights cancel out and what remains is a continuous function on the corresponding intervals and  the supremum becomes a maximum. The last equality is simply by definition. Adding up these relations \eqref{contraction} follows with the value prescribed on $\gamma_2$. In the sequel, we assume that $m \geq 3$ and, therefore, $\overline{m} > 1$.

  Now, let us see why $\vec{T}$ is not contractive when $m=3$. Proceeding as above with $\vec{f} = (f_1,f_2,f_3),\; \vec{g} = (g_1,g_2,g_3) \in  {C}^+({\vec{\Delta}})$, and $\vec{\Delta} = (\Delta_1,\Delta_2,\Delta_3)$, we get
  \begin{align*}
  	d(\vec{T}\vec{f},\vec{T}\vec{g}) =& \max\left\{  \|\ln (|T_1\vec{f}|/|T_1\vec{g}|)\|_{\Delta_2},   \|\ln (|T_2\vec{f}|/|T_2\vec{g}|)\|_{\Delta_1\cup \Delta_3},  \|\ln (|T_3\vec{f}|/|T_3\vec{g}|)\|_{\Delta_2} \right\}\\
  	                              \leq& \max\left\{  \|\ln (|T_1\vec{f}|/|T_1\vec{g}|)\|_{\Delta_1},   \|\ln (|T_2\vec{f}|/|T_2\vec{g}|)\|_{\Delta_2},  \|\ln (|T_3\vec{f}|/|T_3\vec{g}|)\|_{\Delta_3} \right\}\\
  	                                 =& \frac{1}{2}\max\left\{  \|\ln (f_2/g_2)\|_{\Delta_1},   \|\ln (f_1 /g_1 ) + \ln ( f_3/ g_3)\|_{\Delta_2},  \|\ln (f_2/g_2)\|_{\Delta_3} \right\} \\
  	                              \leq& \frac{1}{2}\max\left\{ d(\vec{f},\vec{g}),   \|\ln (f_1 /g_1 ) + \ln ( f_3/ g_3)\|_{\Delta_2},  d(\vec{f},\vec{g}) \right\}.
  \end{align*}

  The problem comes with the norm appearing in the last line because the natural way to bound it is to use the triangle inequality obtaining
  \[
    \|\ln (f_1 /g_1 ) + \ln ( f_3/ g_3)\|_{\Delta_2} \leq \|\ln (f_1 /g_1 )\|_{\Delta_2} + \|\ln ( f_3/ g_3)\|_{\Delta_2} \leq 2 d(\vec{f},\vec{g}).
  \]
  Putting everything together, we get that $T$ is non expansive; that is,
  \begin{equation} \label{non-expansive} d(\vec{T}\vec{f},\vec{T}\vec{g}) \leq d( \vec{f}, \vec{g}).
  \end{equation}
  (The previous arguments allow to deduce that $\vec{T}$ is non expansive for any $m\geq 3$.)

  Let us see that equality is actually attained.  Indeed, taking $\vec{f} \equiv (e,e,e)$, $\vec{g} \equiv (1,1,1)$, where $e$ denotes Euler's constant, the vector of harmonic functions
  \[
    \left(\ln (|T_1\vec{f}|/|T_1\vec{g}|), \ln (|T_2\vec{f}|/|T_2\vec{g}|) , \ln (|T_3\vec{f}|/|T_3\vec{g}|)\right)
  \]
  in $\prod_{k=1}^3(\overline{\CC} \setminus \Delta_k)$
 reduces to the constant vector function $(1/2,1,1/2)$  in which case
 \[ d(\vec{T}\vec{f},\vec{T}\vec{g}) = 1 = d( \vec{f}, \vec{g}).
 \]
 Similarly, one can prove that $\vec{T}$ is not contractive for any $m \geq 3$. Therefore, there is an error  in the statement of \cite[Proposition 1.1]{sasha99}. However, the conclusion that $\vec{T}$ has a unique fixed point is correct since \eqref{contraction} is true as we show next.

 We return to the general case. Take two functions $\vec{f} = (f_1,\ldots,f_m), \vec{g} = (g_1,\ldots,g_m) \in {C}^+( \vec{\Delta}) $. For simplicity in writing, in the sequel we use the notation
 \[ \ln \vec{f} := (\ln f_1,\ldots, \ln f_m), \qquad  { \vec{f}}/{ \vec{g}} := (f_1/g_1,\ldots,f_m/g_m),
 \]
 and
 \[ |{ \vec{f}}/{ \vec{g}}| := (|f_1/g_1|,\ldots,|f_m/g_m|), \qquad \|\ln({ \vec{f}}/{ \vec{g}} )\|_{\vec{\Delta}} := d(\vec{f},\vec{g}).
 \]

According to the definition of the metric $d$, we have
 \begin{equation}
 	\label{formulad}
 	d( \vec{T}^{\overline{m}}\vec{f}, \vec{T}^{\overline{m}} \vec{g} ) =  \| \ln (|\vec{T}^{\overline{m}} \vec{f}|/|\vec{T}^{\overline{m}} \vec{g}| )\|_{\vec{\Delta}}.
 \end{equation}
 Indeed, the $k$-th component $T_k \vec{T}^{\overline{m}-1}\vec{f} $ of $\vec{T}^{\overline{m}}\vec{f}$ is a Szeg\H{o} function in $\overline{\CC} \setminus \Delta_k$ which is real and positive on $\RR \setminus \Delta_k$. Consequently,
 \[
   \ln T_k \vec{T}^{\overline{m}-1}\vec{f}= \ln |T_k \vec{T}^{\overline{m}-1}\vec{f}|\qquad  \mbox{on}\qquad {\RR \setminus \Delta_k}.
 \]
 However, $\ln |T_k \vec{T}^{\overline{m}-1}\vec{f}|$ extends to a harmonic function on $\overline{\CC} \setminus \Delta_k$ and
 \[
   \| \ln  (|T_k \vec{T}^{\overline{m}-1}\vec{f}|/|T_k \vec{T}^{\overline{m}-1}\vec{g}|) \|_{(\Delta_{k-1} \cup \Delta_{k+1})} = \| \ln  |T_k \vec{T}^{\overline{m}-1}\vec{f}| -  \ln  |T_k \vec{T}^{\overline{m}-1}\vec{g}| \|_{(\Delta_{k-1} \cup \Delta_{k+1})}.
 \]
  Consequently, the norm $\|\cdot\|_{\vec{\Delta}}$ of $\ln (\vec{T}^{\overline{m}}\vec{f}/ \vec{T}^{\overline{m}}  \vec{g})$ coincides with the norm of $\ln (|\vec{T}^{\overline{m}} \vec{f}|/|\vec{T}^{\overline{m}} \vec{g}| )$ whose components are the harmonic functions $\ln  (|T_k \vec{T}^{\overline{m}-1}\vec{f}|/|T_k \vec{T}^{\overline{m}-1}\vec{g}|)$. We must find these harmonic functions.

 Fix $k =1,\ldots,m$. For each $u_k \in \mathcal{C}({\Delta_{k-1} \cup \Delta_{k+1}})$ (the space of continuous functions on ${\Delta_{k-1} \cup \Delta_{k+1}}$), we define $P_{j,k}(u_k)$, $j=k-1,k+1$, to be the harmonic function on $\overline{\CC} \setminus \Delta_{j}$ whose boundary values on $\Delta_j$ coincide with the values of $u_k$ on $\Delta_j$ (recall that $\Delta_0 = \Delta_{m+1} = \varnothing$). Let us introduce the matrix map
 \begin{equation} \label{P}
    P := \left(
    \begin{array}{cccccc}
    0 & P_{1,2} & 0 & \cdots & 0 & 0 \\
    P_{2,1}& 0 & P_{2,3} & \cdots & 0 & 0 \\
    0 & P_{3,2} & 0 & \ddots & 0 & 0 \\
    \vdots & \ddots & \ddots & \ddots& \ddots & \vdots \\
    0 & 0 & 0 & \ddots & 0& P_{m-1,m} \\
    0 & 0 & 0 & \cdots & P_{m,m-1}& 0
    \end{array}
    \right).
 \end{equation}
 (Only the entries in the sub-diagonal and supra-diagonal differ from $0$.) Set $u_k = \ln (f_k/g_k)$.

In the following, $(\cdot)^t$ denotes the transpose of the vector $(\cdot)$. Using \eqref{boundary}, we have
 \[ (\ln |\vec{T}^{\overline{m}} \vec{f}|/|\vec{T}^{\overline{m}}\vec{g}|)^t = \frac{1}{2} P (\ln |\vec{T}^{\overline{m}-1} \vec{f}|/|\vec{T}^{\overline{m}-1}\vec{g}|)^t = \cdots = \]
 \[\frac{1}{2^{\overline{m}-1}} P^{\overline{m}-1} (\ln |\vec{T} \vec{f}/\vec{T}\vec{g}|)^t = \frac{1}{2^{\overline{m}}} P^{\overline{m}}(u_1,\ldots,u_m)^t.
 \]
 This formula allows us to bound $\| \ln (|\vec{T}^{\overline{m}} \vec{f}|/|\vec{T}^{\overline{m}} \vec{g}| )\|_{\vec{\Delta}}$.

To illustrate, let us consider first the cases when $m=3,4$ and thus $\overline{m}=2$. We have
\[  \left(
\begin{array}{ccc}
0 & P_{1,2} & 0 \\
P_{2,1} & 0 & P_{2,3} \\
0 & P_{3,2} & 0
\end{array}
\right)^2
\left(
\begin{array}{c}
u_1 \\
u_2 \\
u_3
\end{array}
\right) =
\left(
\begin{array}{c}
P_{1,2}P_{2,1} u_1 + P_{1,2}P_{2,3}u_3 \\
P_{2,1}P_{1,2} u_2 + P_{2,3}P_{3,2} u_2   \\
 P_{3,2}P_{2,1} u_1 + P_{3,2}P_{2,3} u_3
\end{array}
\right)
\]
and
\[  \left(
\begin{array}{cccc}
0 & P_{1,2} & 0 &0 \\
P_{2,1} & 0 & P_{2,3} &0\\
0 & P_{3,2} & 0 & P_{3,4} \\
0 & 0 & P_{4,3} & 0
\end{array}
\right)^2
\left(
\begin{array}{c}
u_1 \\
u_2 \\
u_3 \\
u_4
\end{array}
\right) =
\left(
\begin{array}{c}
P_{1,2}P_{2,1} u_1 + P_{1,2}P_{2,3}u_3 \\
P_{2,1}P_{1,2} u_2 + P_{2,3}P_{3,2} u_2   + P_{2,3}P_{3,4}u_4\\
 P_{3,2}P_{2,1} u_1 + P_{3,2}P_{2,3} u_3  + P_{3,4}P_{4,3} u_3 \\
 P_{4,3}P_{3,2} u_2 + P_{4,3}P_{3,4} u_4
\end{array}
\right).
\]
Using \eqref{formulad}, the maximum principle for harmonic functions, and the triangle inequality, when $m=3$ we get
\[ 4d( \vec{T}^2\vec{f}, \vec{T}^2 \vec{g} ) =\]
\[ \max \{\|P_{1,2}P_{2,1} u_1 + P_{1,2}P_{2,3}u_3\|_{\Delta_2},
\|P_{2,1}P_{1,2} u_2 + P_{2,3}P_{3,2} u_2\|_{\Delta_1 \cup \Delta_3},
 \|P_{3,2}P_{2,1} u_1 + P_{3,2}P_{2,3} u_3\|_{\Delta_2}\}
\]
\[\leq  \max \{\|P_{1,2}P_{2,1} u_1 + P_{1,2}P_{2,3}u_3\|_{\Delta_1},
\|P_{2,1}P_{1,2} u_2 + P_{2,3}P_{3,2} u_2\|_{\Delta_2},
 \|P_{3,2}P_{2,1} u_1 + P_{3,2}P_{2,3} u_3\|_{\Delta_3}\}
\]
\[\leq  \max \{\|P_{1,2}P_{2,1} u_1\|_{\Delta_1} + \|P_{1,2}P_{2,3}u_3\|_{\Delta_1},
\|P_{2,1}P_{1,2} u_2\|_{\Delta_2} + \|P_{2,3}P_{3,2} u_2\|_{\Delta_2},
\]
\[
 \|P_{3,2}P_{2,1} u_1\|_{\Delta_3} + \|P_{3,2}P_{2,3} u_3\|_{\Delta_3}\}
\]

Using the maximum principle for harmonic functions it is easy to deduce that any one of the norms appearing in the last two lines can be bounded by $d(\vec{f},\vec{g})$. For example,
\[
  \|P_{1,2}P_{2,3}u_3\|_{\Delta_1} \leq \|P_{2,3}u_3\|_{\Delta_2} \leq \|\ln (f_3/g_3)\|_{\Delta_3} \leq d(\vec{f},\vec{g})
\]
where the last equality follows from the fact that $\ln (f_3/g_3)$ is one of the components of the vector $(u_1,u_2,u_3)$ whose norm coincides with $d(\vec{f},\vec{g})$. Consequently,
\[ d( \vec{T}^2\vec{f}, \vec{T}^2 \vec{g} ) \leq \frac{1}{2} d(\vec{f},\vec{g})
\]
and $1/2$  is the value assigned to $\gamma_3$.

When $m=4$, reasoning as above you get
\[ 4d( \vec{T}^2\vec{f}, \vec{T}^2 \vec{g} ) \leq 3 d(\vec{f},\vec{g})
\]
which is equivalent to \eqref{contraction} with $\gamma_4 = 3/4$. The $3$ in the previous inequality comes from the fact that the second and third entries of $P^2(u_1,\ldots,u_4)^t$ have three terms while the first and last only two.

Following the ideas expressed in the two examples, in order to estimate $\| \ln (|\vec{T}^{\overline{m}} \vec{f}|/|\vec{T}^{\overline{m}} \vec{g}| )\|_{\vec{\Delta}}$ all we have to do is to count the number of terms in each component of the vector $P^{\overline{m}}(u_1,\ldots,u_m)^t $ and divide the largest number by $2^{\overline{m}}$ to get $\gamma_m$.

The component $r$ of $P^{\overline{m}}(u_1,\ldots,u_m)^t $ is the sum of all the expressions  of the form
\[
  P_{\alpha_1,\beta_1}P_{\alpha_2,\beta_2}\cdots P_{\alpha_{\overline{m}},\beta_{\overline{m}}} u_{\beta_{\overline{m}}}
\]
satisfying the  rules:
\begin{itemize}
\item $\alpha_1 = r$,
\item $\beta_i = \alpha_i \pm 1, 1 \leq \beta_i \leq m, i=1,\ldots,{\overline{m}}$,
\item $\alpha_{i+1} = \beta_i, i=1,\ldots,{\overline{m}}-1$.
\end{itemize}
If this is so, notice that using the maximum principle for harmonic functions and \eqref{boundary} we have
\begin{multline*}
   \|P_{\alpha_1,\beta_1}P_{\alpha_2,\beta_2}\cdots P_{\alpha_{\overline{m}},\beta_{\overline{m}}} u_{\beta_{\overline{m}}}\|_{\Delta_{r-1} \cup \Delta_{r+1}} \leq \|P_{\alpha_1,\beta_1}P_{\alpha_2,\beta_2}\cdots P_{\alpha_{\overline{m}},\beta_{\overline{m}}} u_{\beta_{\overline{m}}}\|_{\Delta_{\alpha_2}}\\
    = \|P_{\alpha_2,\beta_2}\cdots P_{\alpha_{\overline{m}},\beta_{\overline{m}}} u_{\beta_{\overline{m}}}\|_{\Delta_{\alpha_2}} \leq \|P_{\alpha_2,\beta_2}\cdots P_{\alpha_{\overline{m}},\beta_{\overline{m}}} u_{\beta_{\overline{m}}}\|_{\Delta_{\beta_2}} \leq \cdots \leq	\\
	\|P_{\alpha_{\overline{m}},\beta_{\overline{m}}} u_{\beta_{\overline{m}}}\|_{\Delta_{\alpha_{\overline{m}}}} \leq  \|P_{\alpha_{\overline{m}},\beta_{\overline{m}}} u_{\beta_{\overline{m}}}\|_{\Delta_{\beta_{\overline{m}}}} = \|u_{\beta_{\overline{m}}}\|_{\Delta_{\beta_{\overline{m}}}} \leq d(\vec{f},\vec{g}).
\end{multline*}

The assertion of the previous sentence is a consequence of the definition of the product of two matrices and the structure of the matrix $P$. To prove it we introduce some terminology. We say that an expression of the form $P_{\alpha_1,\beta_1}P_{\alpha_2,\beta_2}\cdots P_{\alpha_k,\beta_k}$ is a chain with $k$ links. Only the entries of the form $P_{k,k\pm1}$, where $k+1 \leq m, k-1 \geq 1$, of $P$ are different from zero. The first subindex indicates the row and the second subindex the column where they are located. When we multiply $P$ by itself the entry $P_{k,k+1}$ with $k + 1 \leq m$ gets multiplied by all the entries of $P$ in  row $k +1$, but the only products that survive (are different from zero) are the chains with two links $P_{k,k+1}P_{k+1,k+2}$ (provided $k+2 \leq m$) and $P_{k,k+1}P_{k+1,k}$ which are located in $P^2$ in the positions $(k,k+2)$ and $(k,k)$, respectively. The entry $P_{k,k-1}$ with $k - 1 \geq 1$ gets multiplied by all the entries of $P$ in  row $k -1$, but the only products that survive   are the chains with two links $P_{k,k-1}P_{k-1,k}$ and $P_{k,k-1}P_{k-1,k-2}$ (provided $k-2 \geq 1$) which are located in $P^2$ in the positions $(k,k)$ and $(k,k-2)$, respectively.
We can go on in the same manner arriving to the rules stated above.

Summarizing, whenever $2 \leq \beta_i \leq  {m}-1$, the link $P_{\alpha_i,\beta_i}$ may be followed by $P_{\beta_i,\beta_i +1}$ or $P_{\beta_i,\beta_i -1}$; otherwise, there is only one option, either $P_{1,2}$ when $\beta_i = 1$ or $P_{m,m-1}$ when $\beta_i = m$. Consequently, the central components of  $P^{\overline{m}}(u_1,\ldots,u_m)^t $ are the ones with the largest number of terms since $1$ and $m$ are furthest away from the initial value $r$ determined by the component in question.

When $m$ is odd, the central component is the ${\overline{m}}$-th one. In that case, the number of terms in the ${\overline{m}}$-th component of  $P^{\overline{m}}(u_1,\ldots,u_m)^t $ equals $2^{\overline{m}} -2$ because  there are only two chains with ${\overline{m}}-1$ links  (namely $P_{{\overline{m}},{\overline{m}}-1}P_{{\overline{m}}-1,{\overline{m}}-2}\cdots P_{2,1}$ and $P_{{\overline{m}},{\overline{m}}+1}P_{{\overline{m}}+1,{\overline{m}}+2}\cdots P_{m-1,m}$), which can be 
extended with only one link. Therefore, $\gamma_m = (2^{\overline{m}} -2)/2^{\overline{m}}$.

If $m$ is even the central components are the ${\overline{m}}$ and ${\overline{m}}+1$ ones. In the first case,   only one chain with ${\overline{m}}-1$ links can be extended with only one link, namely $P_{{\overline{m}},{\overline{m}}-1}P_{{\overline{m}}-1,{\overline{m}}-2}\cdots P_{2,1}$. In the second case, the only chain  with ${\overline{m}}-1$ links which can be extended with only one link is $P_{{\overline{m}}+1,{\overline{m}}+2}P_{{\overline{m}}+2,{\overline{m}}+3}\cdots P_{m-1,m}$. Thus, the total number of terms in these two components is $2^{\overline{m}} -1$. This implies that $\gamma_m = (2^{\overline{m}} -1)/2^{\overline{m}}$. With this we conclude the proof.
 \hfill $\Box$

 Incidentally, for any fixed $m$ the maps $T^{{\overline{m}}'}, {\overline{m}}' \geq 2,$ can be defined. They are contractive if and only if ${\overline{m}}' \geq {\overline{m}}$. The constants of contraction are (not necessarily strictly) decreasing  with ${\overline{m}}'$ and tend to zero as ${\overline{m}}' \to \infty$.

 \subsection{Proof of Theorem \ref{th1}}

 First, we will deal with the asymptotics of the polynomials $Q_{\n,j}$, $j=1,\ldots,m$ and then we will obtain the asymptotics of the ML Hermite-Padé polynomials $a_{\n,j}$.

 In order to obtain the strong asymptotics of the polynomials $Q_{\n,j}$ we consider the particular case of the operator $\vec{T}_{\vec{w}}$ for the vector weight
 \begin{equation}
     \label{weight:dir}
     \vec{w}(x) = (\sqrt{(b_1-x)(x-a_1)}\tilde{h}_1(x)\sigma_1'(x),\ldots,\sqrt{(b_m-x)(x-a_m)}\tilde{h}_m(x)\sigma_m'(x))
 \end{equation}
 where $\tilde{h}_m\equiv 1$ and the functions $\tilde{h}_j$, $j=1,\ldots,m-1$ are given by \eqref{def:htil}. Since $\vec{w}$ is fixed throughout this subsection, to simplify the notation, from now on $\vec{T}_{\vec{w}} = \vec{T}$.

 Let $\vec{\mathsf{G}} = (\mathsf{G}_1,\ldots,\mathsf{G}_m)$ be the fixed point of the operator $\vec{T} = (T_1,\ldots,T_m)$. As  above, let $\mathcal{P}_{\n,k}$ be the set of all monic polynomials   of degree $\eta_{\n,k}$
 with real coefficients whose zeros lie in  $\mathbb{C} \setminus (\Delta_{k-1} \cup \Delta_{k+1}) $. As usual $\Delta_0 = \Delta_{m+1} \equiv \varnothing$. Set
 \[  H^{+,\n} :=
   \left\{ \left(\frac{\mathsf{G}_1(\infty)P_{\n,1}}{\Phi_1^{\eta_{\n,1}}},\cdots,\frac{\mathsf{G}_m(\infty)P_{\n,m}}{\Phi_m^{\eta_{\n,m}}}\right) : P_{\n,k}\in \mathcal{P}_{\n,k}, \, k=1,\ldots,m \right\}.
 \]
 Notice that $\frac{\mathsf{G}_k(\infty)P_{\n,k}(x)}{\Phi_k^{\eta_{\n,k}}(x)} > 0, x \in  \Delta_{k-1} \cup \Delta_{k+1}$.

 Let $\tilde{T}_{\n,k}$, $k=1,\ldots,m$, be  the operators defined on $\mathcal{P}_{\n} = \mathcal{P}_{\n,1}\times\cdots\times\mathcal{P}_{\n,m}$ which determine the components of $\tilde{T}_{\n}$; namely,
 \[  \tilde{T}_{\n} = (\tilde{T}_{\n,1},\ldots,\tilde{T}_{\n,m}).
 \]
 Define $T_{\n}:H^{+,\n} \to H^{+,\n} $ where
 \begin{multline*}
 	T_{\n}\left(\frac{\mathsf{G}_1(\infty)P_{\n,1}}{\Phi_1^{\eta_{\n,1}}},\cdots,\frac{\mathsf{G}_m(\infty)P_{\n,m}}{\Phi_m^{\eta_{\n,m}}}\right) =\\
 	\left(\mathsf{G}_1(\infty)\frac{\tilde{T}_{\n,1}(P_{\n,1},\ldots,P_{\n,m})}{\Phi_1^{\eta_{\n,1}}},\cdots,\mathsf{G}_m(\infty)\frac{\tilde{T}_{\n,m}(P_{\n,1},\ldots,P_{\n,m})}{\Phi_m^{\eta_{\n,m}}}\right).
 \end{multline*}
 Notice that any fixed point of $T_{\n}$ generates a fixed point of $\tilde{T}_{\n}$ and reciprocally any fixed point of $\tilde{T}_{\n}$ generates a fixed point of ${T}_{\n}$. The continuity of $T_{\n}$ follows from the continuity of $\tilde{T}_{\n}$.

 \begin{proof}
 	Let $\vec{g}_0^{(\n)}$ be the fixed point of $T_{\n}$. This means that
 \[  \vec{g}_0^{(\n)} =   \left(\frac{\mathsf{G}_1(\infty)Q_{\n,1}}{\Phi_1^{\eta_{\n,1}}},\cdots,\frac{\mathsf{G}_m(\infty)Q_{\n,m}}{\Phi_m^{\eta_{\n,m}}}\right)
 \]
 where $(Q_{\n,1},\ldots,Q_{\n,m})$ is the unique fixed point of the operator $\tilde{T}_{\n}$, see Lemma \ref{unicoQn}. We will prove that
 	\begin{equation}
 		\label{fixed}
 		\lim_{\n\in\Lambda}\|\vec{g}_0^{(\n)} - \vec{\mathsf{G}}\|_{ \vec{\Delta}} =0.
 	\end{equation}
 	Recall that the components of $\vec{\mathsf{G}}$ are Szeg\H{o} functions on $\Omega_k = \overline{\mathbb{C}} \setminus \Delta_k
 $, $k=1,\ldots,m$.

 In order to prove this, we will show that any neighborhood of $\vec{\mathsf{G}}$ in the $\|\cdot\|_{ \vec{\Delta}}$ norm contains a fixed point of the operator $T_{\n}$ for all $|\n| \geq n_0$, $\n \in \Lambda$. This is done using Brouwer's fixed point theorem.

 Fix $\delta > 0$. For each $k=1,\ldots,m$ set
 \[ \Delta_{k,\delta} := \{z \in \mathbb{C}: \min_{x \in \Delta_k} |z - x| \leq \delta\}.
 \]
 Take $\delta $ sufficiently small so that
 \[ \Delta_{k,\delta} \cap (\Delta_{k-1} \cup \Delta_{k+1}) =  \varnothing.
 \]
 Recall that $\Delta_0 = \varnothing= \Delta_{m+1}$. Set
 \[ \Omega_{k,\delta} := \overline{\mathbb{C}} \setminus \Delta_{k,\delta},
 \]
 which obviously contains $\Delta_{k-1} \cup \Delta_{k+1}$. Denote
 \[ \vec{\Omega}_{\delta} := (\Omega_{1,\delta},\ldots,\Omega_{m,\delta}).
 \]
 	
 	Let $H^+(\vec{\Omega}_{\delta})$ denote the cone of all $m$-tuples $(g_1,\ldots,g_m)$ of vector functions such that $g_k$, $k=1,\ldots,m,$ is holomorphic in $\Omega_{k,\delta}$, symmetric (with respect to the real line), and positive on $\Omega_{k,\delta}\cap\RR$. For $\vec{g} = (g_1,\ldots,g_m) \in H^+(\vec{\Omega}_{\delta})$, we define
 	\[ \|\vec{g}\|_{\vec{\Omega}_{\delta}} := \max\{\sup\{|g_k(z)|: z\in \Omega_{k,\delta}\}:k=1,\ldots,m\}\,\,(\leq \infty)
 	\]
 	and
 	\[  \min_{\vec{\Delta}} \vec{g} := \min \{\min  \{g_k(z): z\in \Delta_{k-1}\cup\Delta_{k+1}\}: k=1,\ldots,m\},
 	\]
 	where $\Delta_0=\Delta_{m+1}=\varnothing$.
 	
 	Fix a constant $C\geq 1$ so that $C \geq 2 \|\vec{\mathsf{G}}\|_{\vec{\Omega}_{\delta}} $ and $ C^{-1} \leq \frac{1}{2} \min_{\vec{\Delta}} \vec{\mathsf{G}}$. Define
 	\[  H^+(\vec{\Omega}_{\delta}, C) := \{\vec{g}\in H^+(\vec{\Omega}_{\delta}): \|\vec{g}\|_{\vec{\Omega}_{\delta}}\leq C,\; \min_{\vec{\Delta}} \vec{g}\geq C^{-1}\}.
 	\]
 	Obviously, $\vec{\mathsf{G}}\in H^+(\vec{\Omega}_{\delta}, C)$.
 	

 Let $(\vec{\rho}_n )_{n \in \mathbb{N}}$ be an arbitrary sequence of elements in $ H^+(\vec{\Omega}_{\delta}, C)$. For each $k = 1,\ldots,m$  the sequence of functions made up by  the $k$-th component of $\vec{\rho}_k$ is uniformly bounded on $\Omega_{k,\delta}$. Therefore, there exists  $\mathbb{I} \subset \mathbb{N}$ such that $(\vec{\rho}_n)_{n \in \mathbb{I}}$ converges componentwise to some vector function $\vec{\rho}$ uniformly on each compact subset of $\vec{\Omega}_{\delta} $, The $k$-th component of $\vec{\rho}$ is, therefore, holomorphic and symmetric on $\Omega_{k,\delta}$, and non-negative on $\Omega_{k,\delta}\cap \mathbb{R}, k=1,\ldots,m$. Additionally,
 	\[
 	\min_{  \vec{\Delta}} \vec{\rho} = \lim_{n \in \mathbb{I}} \min_{\vec{\Delta}} \vec{\rho}_n \geq C^{-1},	
 	\]
 	and
 	\[
 	 \|\vec{\rho}\|_{\vec{\Omega}_{\delta}} = \lim_{n\in \mathbb{I}} \|\vec{\rho}_n\|_{\vec{\Omega}_{\delta}} \leq C.
 	\]
 	Consequently, $\vec{\rho} \in H^+(\vec{\Omega}_{\delta},C)$. We conclude that $H^+(\vec{\Omega}_{\delta},C)$ is compact with the topology of uniform convergence on compact subsets on the space $H(\vec{\Omega}_{\delta})$ of vector analytic functions on $\vec{\Omega}_{\delta}$.

 	Fix an arbitrary $\theta > 0$. Let
 	\[
 	\omega(\theta)= \{ \vec{g} \in H^+(\vec{\Omega}_{\delta}, C):\|\vec{g} - \vec{\mathsf{G}}\|_{\vec{\Delta}}\leq \theta \}.
 	\]
  Analogously, for every $\varepsilon > 0$, set
 	\[
 	\omega_\varepsilon := \{ \vec{g}\in H^+(\vec{\Omega}_{\delta}, C) : d(\vec{g},\vec{\mathsf{G}}) \leq \varepsilon\}.
 	\]
 	There exists $\varepsilon_0$ such that
 	\[
 	\omega_{\varepsilon} \subset \omega(\theta), \qquad 0 < \varepsilon \leq \varepsilon_0,
 	\]
 	for, otherwise, we could find a sequence of vector functions in $H^+(\vec{\Omega}_{\delta}, C) \subset \mathcal{C}^+(\vec{\Delta})$ which converges to $\vec{\mathsf{G}}$ in the $d$ metric but not in the $\|\cdot\|_{ \vec{\Delta}}$ norm which would contradict \eqref{equiv}.
 	
 	Consider
 	\[
 	\omega_{\varepsilon,\n} := \omega_\varepsilon \cap H^{+,\n}.
 	\]
 	Let $\mu_k$ be the representing measure of $\mathsf{G}_k$; that is, $\mathsf{G}_k(z) = \mathsf{G}_k(\mu_k,z)$, $z\in\Omega_k$. Taking into account Lemma \ref{prop2}, in particular \eqref{predet*} and \eqref{equiv}, we get
 	\[
 	\lim_{\n\in\Lambda} d(\vec{\tilde{g}}_1^{(\n)},\vec{\mathsf{G}}) = 0,
 	\]
 	where
 	\[
 	\vec{\tilde{g}}_1^{(\n)} := \left(\frac{\mathsf{G}_1(\infty)\tilde{Q}_{\n,1}}{\Phi_1^{\eta_{\n,1}}}, \cdots, \frac{\mathsf{G}_m(\infty)\tilde{Q}_{\n,m}}{\Phi_m^{\eta_{\n,m}}}\right)
 	\]
 and the $\tilde{Q}_{\n,j}$, $j=1,\ldots,m,$ are the orthonormal polynomials $\tilde{q}_{\n,j}$
 given in \eqref{poltilde} re-normalized to be monic. Because of \eqref{predet*} we conclude that, for every $\varepsilon$, $0<\varepsilon<\varepsilon_0,$ there exists $n_0$ such that $\omega_{\varepsilon,\n}\neq\varnothing$ for $\n\in\Lambda$, $|\n|>n_0$.

Let  $H(\vec{\Omega})$ be the locally convex space  of all vector analytic functions on $\vec{\Omega} = (\Omega_1, \ldots,\Omega_m)$ with the topology of uniform convergence on compact subsets of $\vec{\Omega}$.
 Let us show that for each $\n \in \Lambda$ fixed, $\omega_{\varepsilon,\n}$ is a compact subset of $H(\vec{\Omega})$ with this topology.

 For $\n \in \Lambda$ fixed, consider an arbitrary sequence
 \[  (\rho_{n,1} ,\ldots,\rho_{n,m} )_{n \in \mathbb{N}} := \left(\frac{\mathsf{G}_1(\infty) {P}_{n,1}}{\Phi_1^{\eta_{\n,1}}}, \cdots, \frac{\mathsf{G}_m(\infty) {P}_{n,m}}{\Phi_m^{\eta_{\n,m}}}\right)_{n \in \mathbb{N}}
 \subset \omega_{\varepsilon,\n}.
 \]
 We must show that it has a subsequence which converges to an element of $\omega_{\varepsilon,\n}$.

Fix $k \in \{1,\ldots,m\}$.   Consider the sequence of monic polynomials  with real coefficients  $(P_{n,k})_{n\in \mathbb{N}}$ of degree $\eta_{\n,k}$ whose zeros lie in $ \mathbb{C} \setminus (\Delta_{k-1} \cup \Delta_{k+1})$. Since $\omega_{ \epsilon,\n} \subset H^+(\vec{\Omega}_{\delta},C)$ it readily follows that the zeros of the polynomials $P_{n,k}$
remain uniformly bounded away from $\Delta_{k-1} \cup \Delta_{k+1} \cup \{\infty\}$ in $n\in \mathbb{N}$.
Therefore, there exists a subsequence of indices $I \subset \mathbb{N}$ and a monic polynomial $P_k$ with real coefficients of degree $\eta_{\n,k}$, whose zeros lie on $\mathbb{C} \setminus (\Delta_{k-1} \cup \Delta_{k+1})$, such that $\lim_{n \in I} P_{n,k} = P_k$ uniformly on every compact subset of $\mathbb{C}$.  Since this is true for every $k = 1,\ldots,m$, it readily follows that there exists a subsequence $I' \subset \mathbb{N}$ such that
\[ \lim_{n \in I'}  (\rho_{n,1},\ldots,\rho_{n,m}) = \left(\frac{\mathsf{G}_1(\infty) {P}_{1}}{\Phi_1^{\eta_{\n,1}}}, \cdots, \frac{\mathsf{G}_m(\infty) {P}_{m}}{\Phi_m^{\eta_{\n,m}}}\right) \in  \omega_{\varepsilon,\n},
\]
uniformly on every compact subset of $\vec{\Omega}$.   With this we conclude the proof that $\omega_{\varepsilon,\n}$ is compact as claimed.

Fix $\n \in \Lambda$. Let us show that $\omega_{\varepsilon,\n}$ is convex. Take two arbitrary vector functions in $\omega_{\varepsilon,\n}$
\[ \vec{p}_j =  \left(\frac{\mathsf{G}_1(\infty) {P}_{j,1}}{\Phi_1^{\eta_{\n,1}}}, \cdots, \frac{\mathsf{G}_m(\infty) {P}_{j,m}}{\Phi_m^{\eta_{\n,m}}}\right) , \qquad j=1,2.
\]
 Let us show that for every $\beta, 0 \leq \beta \leq 1$
\[ \beta \vec{p}_1 + (1 - \beta) \vec{p}_2  \in \omega_{\varepsilon,\n}.
\]
Obviously, $\beta \vec{p}_1 + (1 - \beta) \vec{p}_2 \in H^{+,\n} \cap H^+(\vec{\Omega}_{\delta}, C)$.   In particular, for each $k = 1,\ldots,m,$
\begin{equation} \label{positive} {\mathsf{G}_k(\infty) (\beta {P}_{1,k}(x)+(1-\beta)P_{2,k}(x))}{\Phi_k^{-\eta_{\n,k}}(x)} > 0, \qquad x \in \Delta_{k-1} \cup \Delta_{k+1}.
\end{equation}
It remains to prove that
\begin{equation} \label{desig*}
\left\| \ln \left( \frac{\mathsf{G}_k(\infty) (\beta {P}_{1,k}+(1-\beta)P_{2,k})}{\Phi_k^{\eta_{\n,k}}G_k} \right)\right\|_{\Delta_{k-1} \cup \Delta_{k+1}} \leq \varepsilon.
\end{equation}
For each $x \in \Delta_{k-1} \cup \Delta_{k+1}$ fixed, $\beta {P}_{1,k}(x)+(1-\beta)P_{2,k}(x)$ is an intermediate
point between ${P}_{1,k}(x)$ and ${P}_{2,k}(x)$. This fact, together with \eqref{positive}, implies that for $x \in \Delta_{k-1} \cup \Delta_{k+1}$
\[ 0 <  \min \left\{ \frac{\mathsf{G}_k(\infty) ({P}_{1,k}(x)) }{\Phi_k^{\eta_{\n,k}(x)}G_k(x)}, \frac{\mathsf{G}_k(\infty) ({P}_{2,k}(x)) }{\Phi_k^{\eta_{\n,k}(x)}G_k(x)} \right\} \leq
\]
\[ \frac{\mathsf{G}_k(\infty) (\beta {P}_{1,k}(x)+(1-\beta)P_{2,k}(x))}{\Phi_k^{\eta_{\n,k}(x)}G_k(x)}  \leq \max \left\{ \frac{\mathsf{G}_k(\infty) ({P}_{1,k}(x)) }{\Phi_k^{\eta_{\n,k}(x)}G_k(x)}, \frac{\mathsf{G}_k(\infty) ({P}_{2,k}(x)) }{\Phi_k^{\eta_{\n,k}(x)}G_k(x)} \right\}.
\]
Using the monotonicity of the logarithm \eqref{desig*} readily follows.

Now we need to show that $T_{\n}(\omega_{\varepsilon, \n})\subset \omega_{\varepsilon, \n}$ for all $\n \in \Lambda$ with $|\n|$ large enough. Indeed, there exist $0 < \varepsilon' < \varepsilon$ and $n_0$ such that for all $\n \in \Lambda, |\n| \geq n_0$, we have
 \begin{equation}
 	\label{include1}
 	T_{\n}(\omega_{\varepsilon', \n})\subset \omega_{\varepsilon, \n}.
 \end{equation}
 If we assume the contrary, we could find a sequence $(\vec{g}^{(\n)})_{\n \in \Lambda'}$, $\Lambda' \subset \Lambda$,
 \[ \vec{g}^{(\n)} =  \left(\frac{\mathsf{G}_1(\infty) \tilde{Q}_{\n,1}}{\Phi_1^{\eta_{\n,1}}}, \cdots, \frac{\mathsf{G}_m(\infty) \tilde{Q}_{\n,m}}{\Phi_m^{\eta_{\n,m}}}\right)
  \in \omega_{\varepsilon, \n},
 \] such that
 \[
 \lim_{\n \in \Lambda'} d(\vec{\mathsf{G}}, \vec{g}^{(\n)}) = 0, \qquad \mbox{and} \qquad d(\vec{\mathsf{G}},T_{\n}(\vec{g}^{(\n)})) = d(T(\vec{\mathsf{G}}),T_{\n}(\vec{g}^{(\n)}))\geq \varepsilon.
 \]
 Since $ (\vec{g}^{(\n)})_{\n \in \Lambda'} \subset H^+(\vec{\Omega}_{\delta}, C)$ it follows that for every $j=1,\ldots,m$ the zeros of $(\tilde{Q}_{\n,j})_{\n \in \Lambda'}$ remain uniformly bounded away from $\Delta_{j-1} \cup \Delta_{j+1}$. Indeed, if we assume the contrary, taking a convenient subsequence, we would conclude that $\mathsf{G}_j$ must have some zero in $\Delta_{j-1} \cup \Delta_{j+1}$. Therefore, we can apply Lemma \ref{prop5} and from
 \eqref{limfund*}   and \eqref{equiv} it follows that
 \[
 \lim_{\n \in \Lambda'} d(T(\vec{\mathsf{G}}),  {T}_{\n}( \vec{g}^{(\n)})) = 0.
 \]
 which is not possible because the limit value should have to be $\geq \varepsilon$. So, such $\varepsilon' > 0$ exists.

 It remains to prove that
 \begin{equation}
 	\label{include2}
 	T_{\n}(\omega_{\varepsilon, \n} \setminus   \omega_{\varepsilon', \n}) \subset \omega_{\varepsilon, \n}.
 \end{equation}
 for all $\n \in \Lambda$ with $|\n|$ sufficiently large. An inmediate consequence of the maximum principle is that
 the operator $\vec{T}$ is non-expansive, see \cite[Proposition 1.1]{sasha99} or \eqref{non-expansive} for the special  case $m=3$. Therefore,
 \[
 d( \vec{\mathsf{G}}, \vec{T}(\vec{g})) = d(\vec{T}(\vec{\mathsf{G}}), \vec{T}(\vec{g}))  \leq d(\vec{\mathsf{G}},  \vec{g}), \qquad \vec{ g} \in \omega_{\varepsilon}.
 \]
 Consequently,
 \[
 \sup_{\vec{g} \in \omega_{\varepsilon} \setminus   \omega_{\varepsilon'}}
 \frac{d(\vec{T}(\vec{\mathsf{G}}), \vec{T}(\vec{g}))}{d(\vec{\mathsf{G}},  \vec{g})} \leq 1.
 \]
 Consider
 \[ \tilde{\omega}_{\varepsilon,\varepsilon'} = \{\vec{g} \in  \omega_{\varepsilon}\setminus \omega_{\varepsilon'}: g_k(\infty) = \mathsf{G}_k(\infty), k=1,\ldots,m\}.
 \]
 Suppose that
 \[ \sup_{\vec{g} \in  \tilde{\omega}_{\varepsilon,\varepsilon'}}
 \frac{d(\vec{T}(\vec{\mathsf{G}}), \vec{T}(\vec{g}))}{d(\vec{\mathsf{G}},  \vec{g})} = 1.
 \]
  Then, there exist a sequence
  $(\vec{g}_n)_{n \in \mathbb{N}} \subset  \tilde{\omega}_{\varepsilon,\varepsilon'}$ such that
 \[
 \lim_{n \to \infty} \frac{d(\vec{T}(\vec{\mathsf{G}}), \vec{T}(\vec{g}_n))}{d(\vec{\mathsf{G}},  \vec{g}_n )} = 1,
 \]
 a subsequence $(\vec{g}_n )_{n \in \mathbb{I} \subset \mathbb{N}} $, and
  $\vec{g} \in H^{+}(\vec{\Omega}_\delta,C)$ such that
 \[
  \lim_{n \in \mathbb{I}} \|\vec{g}_n  - \vec{g}\|_{\vec{\Omega}_{\delta}} = 0,
  \qquad \lim_{n \in \mathbb{I}} d(\vec{T}(\vec{g}_n  ), \vec{T}( \vec{g})) = 0,
 \]
 and
 \[
 {d(\vec{T}(\vec{\mathsf{G}}), \vec{T}(\vec{g}))} = {d(\vec{\mathsf{G}},  \vec{g} )}.
 \]
 That is
 \begin{equation}
 	\label{igualnormas}
 	\|\ln(\vec{T}(\vec{g})/\vec{T}(\vec{\mathsf{G}}))\|_{\vec{\Delta}} =  \|\ln(\vec{g}/\vec{\mathsf{G}})\|_{\vec{\Delta}}.
 \end{equation}
 We will show that this implies that $\vec{g} = \vec{\mathsf{G}}$ which is not possible because
 $d(\vec{g},\vec{\mathsf{G}}) \geq \varepsilon' >0$.

 Let $\vec{u} := \ln(\vec{g}/\vec{\mathsf{G}}) = (\ln(g_1/\mathsf{G}_1),\ldots,\ln(g_m/\mathsf{G}_m))$ be the vector function whose $k$-th component  $u_k$ is the restriction to $\Delta_{k-1} \cup \Delta_{k+1}$ of $\ln(g_k/G_k)$.  By the definition of the operator $\vec{T}$, we have  $\ln(|\vec{T}(\vec{g})|/|\vec{T}(\vec{\mathsf{G}})|)$ is the vector function whose $j$-th component is the harmonic function on $\Omega_j$ whose  boundary values on $\Delta_j$ coincide with the restriction on $\Delta_j$ of $(u_{j-1} + u_{j+1})/2$. If \eqref{igualnormas} takes place, there exists $k$, $1\leq k\leq m$ such that (recall the definition of the operators $P_{j,k}$ just before the  introduction of the matrix \eqref{P})
 \begin{align*}
 	\|\vec{u}\|_{\vec{\Delta}} =& \frac{1}{2}\|P_{k,k-1}u_{k-1} + P_{k,k+1}u_{k+1}\|_{\Delta_{k-1}\cup\Delta_{k+1}}\\
 	\leq& \frac{1}{2}\left(\|P_{k,k-1}u_{k-1}\|_{\Delta_{k-1}\cup\Delta_{k+1}} + \|P_{k,k+1}u_{k+1}\|_{\Delta_{k-1}\cup\Delta_{k+1}} \right)\\
 	\leq& \frac{1}{2}\left(\|P_{k,k-1}u_{k-1}\|_{\Delta_k} + \|P_{k,k+1}u_{k+1}\|_{\Delta_{k}}\right) \leq \|\vec{u}\|_{\vec{\Delta}}.
 \end{align*}
If $k=1$ or $k=m$,   you actually get $\|\vec{u}\|_{\vec{\Delta}} \leq \frac{1}{2} \|\vec{u}\|_{\vec{\Delta}}$ because in all the intermediate steps there is only one term, so $\|u\|_{\vec{\Delta}} = 0$ and we would be done.

Now, for any other value of $k$ the chain of inequalities implies that there is equality on each step; in particular
 \[
 \|P_{k,k-1}u_{k-1}\|_{\Delta_{k-1}\cup\Delta_{k+1}} =  \|P_{k,k-1}u_{k-1}\|_{\Delta_{k} } = \|\vec{u}\|_{\vec{\Delta}}, \qquad
 \]
 and
 \[
 \|P_{k,k+1}u_{k+1}\|_{\Delta_{k-1}\cup\Delta_{k+1}} =  \|P_{k,k+1}u_{k+1}\|_{\Delta_{k}} = \|\vec{u}\|_{\vec{\Delta}}.
 \]
 We conclude that $P_{k,k-1}u_{k-1}$ and $P_{k,k+1}u_{k+1}$ are constant because they are harmonic functions in $\Omega_k$ which attain their maximum at interior points. The constant value they take is $\|\vec{u}\|_{\vec{\Delta}}$ (or  $-\|\vec{u}\|_{\vec{\Delta}}$). In turn this implies that $u_{k+1}$ and $u_{k-1}$ restricted to $\Delta_k$ are constantly equal to $\|\vec{u}\|_{\vec{\Delta}}$ (or  $-\|\vec{u}\|_{\vec{\Delta}}$).  Consequently,
 \[
 \ln(g_{k+1}(x)/\mathsf{G}_{k+1}(x)) = \ln(g_{k-1}(x)/\mathsf{G}_{k-1}(x)) = \|\vec{u}\|_{\vec{\Delta}}, \qquad x \in \Delta_k.
 \]
 However, $\ln(g_{k+1} /\mathsf{G}_{k+1} ) \in H^+(\Omega_{\delta,k+1}), \ln(g_{k-1} /\mathsf{G}_{k-1} ) \in H^+(\Omega_{\delta,k-1}) $. Since they are constant on the interval $\Delta_k$ of their regions of holomorphy, they are constantly equal to $\|u\|_{\vec{\Delta}}$.
 Now,
 \[
 g_{k+1}(\infty) =   \mathsf{G}_{k+1}(\infty) \qquad \mbox{and} \qquad
   g_{k-1}(\infty)   =   \mathsf{G}_{k-1}(\infty).
 \]
 It follows that
 \[
 \|\vec{u}\|_{\vec{\Delta}} = \ln(g_{k+1}(\infty)/\mathsf{G}_{k+1}(\infty))  =  \ln(g_{k-1}(\infty)/\mathsf{G}_{k-1}(\infty)) = 0
 \]
 Therefore, $\vec{g} \equiv \vec{\mathsf{G}}$, which is impossible as pointed out above. This means that
 \[
 \sup_{\vec{g} \in \tilde{\omega}_{\varepsilon,\varepsilon'}} \frac{d(\vec{T}(\vec{\mathsf{G}}), \vec{T}(\vec{g}))}{d(\mathsf{G},  \vec{g})} = \gamma < 1.
 \]
 In other words,
 \begin{equation}
 	\label{gamma}
 	d(\vec{T}(\vec{\mathsf{G}}), \vec{T}(\vec{g})) \leq \gamma d(\vec{ \mathsf{G}},  \vec{g}), \qquad \vec{g} \in \tilde{\omega}_{\varepsilon,\varepsilon'}.
 \end{equation}

Now, we are going to show that there exists $n_0$ such that for all $\n \in \Lambda$ with $|\n|>n_0$ and $\vec{g}\in\omega_{\varepsilon, \n}$, we have
 	\begin{equation}
 		\label{desig}
 		d(T_{\n}(\vec{g}), \vec{T} (\vec{g}))< (1-\gamma)\varepsilon.
 	\end{equation}
 	
 Indeed, if we assume the contrary, we can find  $(\vec{g}^{(\n)})_{\n\in\Lambda'}$, $\Lambda' \subset \Lambda$, $\vec{g}^{(\n)} \in \omega_{\varepsilon,\n}$ such that
 	\begin{equation}
 		\label{contra}
 		d(T_{\n}(\vec{g}^{(\n)} ), \vec{T} (\vec{g}^{(\n)}) )\geq  (1-\gamma)\varepsilon.
 	\end{equation}
 	Since $\omega_{\varepsilon, \n}\subset H^+(\vec{\Omega}_{\delta}, C)$   the sequence $(\vec{g}^{(\n)})_{\n\in\Lambda'}$ is uniformly bounded in the norm $\|\cdot\|_{\vec{\Omega}_\delta}$. Therefore, there exists a function $\vec{g}\in H^+(\vec{\Omega}_{\delta}, C)$ and a subsequence $\mathbf{\Lambda}''\subset \Lambda'$ such that
 	\[
 	  \lim_{\n\in\Lambda''} \|\vec{g}^{(\n)}-\vec{g}\|_{\vec{\Omega}_\delta} = 0.
 	\]
 	In particular,
 	\[
 	  \lim_{\n\in\Lambda''} \|\vec{g}^{(\n)}-\vec{g}\|_{ \vec{\Delta}} = 0,
 	\]
 and \eqref{limden} in Lemma \ref{prop5} takes place with $\vec{f}_{\vec{n}} = \vec{g}^{(\n)}$, $\vec{f} = \vec{g}$, and $\Lambda = \Lambda''$. For each $j=1,\ldots,m$, the zeros of the $j$-th component of the vector
 functions $\vec{g}^{(\n)}$ remain uniformly bounded away from $\Delta_{j-1} \cup \Delta_{j+1}$ for $\n \in \Lambda''$ . Indeed, the contrary  implies  that the $j$-th component of $\vec{g}$ has a zero on $\Delta_{j-1} \cup \Delta_{j+1}$ which is not the case. 	According to \eqref{limfund*}
 	\[
 	  \lim_{\n\in\mathbf{\Lambda}''} \|T_{\n}(\vec{g}^{(\vec{n})}) -  \vec{T} (\vec{g})\|_{ \vec{\Delta}} =0
 	\]
 	which contradicts \eqref{contra} due to \eqref{equiv} and the continuity of the operator $\vec{T}$.

  For $\vec{g} \in \omega_{\varepsilon,\n} \setminus \omega_{\varepsilon',\n} \subset \tilde{\omega}_{\varepsilon,\varepsilon'}$, on account of \eqref{gamma} and \eqref{desig}, we obtain
 \[
 d(\vec{\mathsf{G}},T_{\n}(\vec{g})) \leq d(\vec{T}(\vec{\mathsf{G}}), \vec{T}(\vec{g})) + d(\vec{T}(\vec{g}),T_{\n}(\vec{g})) \leq \gamma d(\vec{ \mathsf{G}},  \vec{g}) + (1 - \gamma)\varepsilon < \varepsilon.
 \]
 This implies \eqref{include2} which together with \eqref{include1} give that there exists $n_0$ such that
 \[
   T_{\n}(\omega_{\varepsilon, \n} ) \subset \omega_{\varepsilon, \n}, \qquad \n \in \Lambda, \qquad |\n| \geq n_0.
 \]

 We have shown that for each fixed $\n \in \Lambda$, $|\n| > n_0$, $\omega_{\varepsilon,\n}$ is a non-empty convex, compact subset of
 the locally convex space $H(\vec{\Omega})$ of vector analytic functions on $\vec{\Omega} = (\Omega_1, \ldots,\Omega_m)$ with the topology of uniform convergence on compact subsets of $\vec{\Omega}$, and $T_{\n}$ is a continuous operator which maps $\omega_{\varepsilon,\n} \subset H(\vec{\Omega})$ into itself; therefore, Brouwer's fixed point theorem asserts that  $T_{\n}$ has at least one fixed point in $\omega_{\varepsilon,\n}$ (see \cite[Theorem V.19]{ReedSimon}). This combined with Lemma \ref{unicoQn} gives us that the   fixed point is unique and  equals
 \[
   \left(\frac{\mathsf{G}_1(\infty)Q_{\n,1}}{\Phi_1^{\eta_{\n,1}}},\cdots,
   \frac{\mathsf{G}_m(\infty)Q_{\n,m}}{\Phi_m^{\eta_{\n,m}}}\right) \in \omega_{\varepsilon,\n},
 \]
 where  $(Q_{\n,1},\ldots,Q_{\n,m})$ is the vector  whose components  are the polynomials whose roots coincide with the zeros of the forms $\mathcal{A}_{\n,j}$, $j=1,\ldots,m$. Since $\varepsilon$ can be taken arbitrarily small \eqref{fixed} follows and the statements \eqref{limfund***} and \eqref{kappas} of Theorem \ref{th1} hold using again Lemma \ref{prop5}.

 From \cite[Prop. 1.2]{lysov20} and \cite[Th 1.1]{SLL22}, we know that for $j=0,\ldots,m-1$
\[
  \lim_{\n\in\Lambda} \frac{a_{\n,j}(z)}{a_{\n,m}(z)} = \widehat{s}_{m,j+1}(z)
\]
uniformly on compact subsets of $\Omega_m$. Since $a_{\n,m}$ coincides with $Q_{\n,m}$ we have
\[
  \lim_{\n\in\Lambda} \frac{a_{\n,j}(z)}{\Phi_m^{|\n|}(z)}\frac{\Phi_m^{|\n|}(z)}{Q_{\n,m}(z)} =\widehat{s}_{m,j+1}(z).
\]
Now, using  \eqref{limfund***},  \eqref{MLasin} follows and we conclude the proof of Theorem 1.4. \end{proof}

With the asymptotic behavior of the vector sequence $(Q_{\n,1},\ldots,Q_{\n,m}), \n \in \Lambda,$ it is easy to derive the asymptotics of the linear forms $\mathcal{A}_{\n,j}$, $j=0,\ldots,m-1$.

\begin{corollary}
	\label{asym_forms}
	Suppose that the assumptions of Theorem \ref{th1} hold and the $\mathcal{A}_{\n,j}$, $j=0,\ldots,m-1$, are defined by \eqref{df:HP:ly:1}. Let $\varepsilon_{\n,j+1}$, $j=0,\ldots,m-1$ denote the constant sign which the varying measure $\frac{\mathcal{H}_{\n,j+1}(x) \D \sigma_{j+1}(x)}{Q_{\n,j}(x)Q_{\n,j+2}(x)}$ adopts on $\Delta_{j+1}$. Then,
	\begin{equation}
		\label{asym_An0}
		\lim_{\n\in\Lambda}  {\varepsilon_{\n,1}K_{\n,0}^2\Phi_{1}^{\eta_{\n,1}}(z)\mathcal{A}_{\n,0}(z)}  = \frac{\mathsf{G}_{1}(\infty)}{\mathsf{G}_{1}(z)} \frac{1}{\sqrt{(z-a_{1})(z-b_{1})}},
	\end{equation}
  uniformly on compact subsets of $\overline{\CC}\setminus\Delta_1$, and for $j=1,\ldots,m-1$
	\begin{equation}
		\label{asym_Anj}
	\lim_{\n\in\Lambda} \frac{\varepsilon_{\n,j+1}K_{\n,j}^2\Phi_{j+1}^{\eta_{\n,j+1}}(z)\mathcal{A}_{\n,j}(z)}{\Phi_j^{\eta_{\n,j}}(z)} = \frac{\mathsf{G}_{j}(z)}{\mathsf{G}_{j}(\infty)} \frac{\mathsf{G}_{j+1}(\infty)}{\mathsf{G}_{j+1}(z)} \frac{1}{\sqrt{(z-a_{j+1})(z-b_{j+1})}},
	\end{equation}
	uniformly on compact subsets of $\overline{\CC}\setminus(\Delta_j \cup \Delta_{j+1})$.  	
\end{corollary}
\begin{proof}
	From \eqref{formrec2} and the definition of $\mathcal{H}_{\n,j}$ we have
	\[
	\mathcal{A}_{\n,j}(z) = \frac{Q_{\n,j}(z)}{Q_{\n,j+1}(z)} \int \frac{Q_{\n,j+1}^2(x)}{z-x}\frac{\mathcal{H}_{\n,j+1}(x)\D\sigma_{j+1}(x)}{Q_{\n,j}(x)Q_{\n,j+2}(x)}, \qquad j=0,\ldots,m-1,
	\]
	where the equality holds in $\Omega_{j+1}$. Due to \eqref{formrec3},
	\[
	  \frac{\varepsilon_{\n,j+1}K_{\n,j}^2\mathcal{A}_{\n,j}(z)}{(\Phi_j^{\eta_{\n,j}}/\Phi_{j+1}^{\eta_{\n,j+1}})(z)} = \frac{Q_{\n,j}(z)}{\Phi_j^{\eta_{\n,j}}(z)}\frac{\Phi_{j+1}^{\eta_{\n,j+1}}(z)}{Q_{\n,j+1}(z)}h_{\n,j}(z).
	\]
Now, from Lemma \ref{lem:hnj} and Theorem \ref{th1}, \eqref{asym_Anj} is immediate. For \eqref{asym_An0} take into account that $Q_{\n,0}\equiv 1, \Phi_0 \equiv 1,$ and $\Delta_0=\varnothing$.
\end{proof}

One may be tempted to think that condition  \eqref{const_ray} on the sequence of multi-indices $\Lambda$ in Theorem \ref{th1} is too restrictive. The fact is that the strong asymptotics of multilevel Hermite-Pad\'e polynomials is very sensitive to the way the sequence of multi-indices is taken. For example, if $\Lambda$ verifies \eqref{const_ray} and we define ${\Lambda}^l$ as
\[
  {\Lambda}^l := \{{\n}^l \in (\mathbb{Z}_+^m)^*: {\n}^l = \n + \vec{e}_l, \n \in \Lambda\},
\]
where $\vec{e}_l$ is the unitary vector with $1$ in $l$-th component, then for each $l$ fixed there exists
\begin{equation}
	\label{lambdal}
	\lim_{\n^l \in \Lambda^l} \frac{Q_{\n^l,j}(z)}{\Phi_{j}^{\eta_{\n^l,j}}(z)}, \qquad j=1,\ldots,m,
\end{equation}
but the limiting functions will depend on $l$. To prove this, write
\[
  \frac{Q_{\n^l,j}(z)}{\Phi_{j}^{\eta_{\n^l,j}}(z)} = \frac{Q_{\n,j}(z)}{\Phi_{j}^{\eta_{\n,j}}(z)}  \frac{ \Phi_{j}^{\eta_{\n,j}}(z)}{\Phi_{j}^{\eta_{\n^l,j}}(z)}\frac{Q_{\n^l,j}(z)}{Q_{\n,j}(z)}.
\]
Since $\Lambda $ verifies \eqref{const_ray}
\[
  \lim_{\n \in \Lambda} \frac{Q_{\n,j}(z)}{\Phi_{j}^{\eta_{\n,j}}(z)} = \frac{\mathsf{G}_j (z)}{\mathsf{G}_j (\infty)},
\]
uniformly on each compact subset of $\overline{\CC} \setminus \Delta_j$.
The second factor satisfies
\[
  \frac{ \Phi_{j}^{\eta_{\n,j}}(z)}{\Phi_{j}^{\eta_{\n^l,j}}(z)} =
\left\{
\begin{array}{cc}
1,  & j = 1,\ldots,l-1 , \\
1/\Phi_{j}(z), & j=l,\ldots,m .
\end{array}
\right.
\]
Finally, using \cite[Theorem 3.4]{SLL22}
\[ \lim_{\n \in \Lambda}  \frac{Q_{\n^l,j}(z)}{Q_{\n,j}(z)} =  \frac{{F}_j^{(l)}(z)}{{F}_j^{(l)}(\infty)},
\]
uniformly on each compact subset of $\overline{\CC} \setminus \Delta_j$, where
$ {F}_j^{(l)}, j=1,\ldots,m,$ is the solution of a system of boundary value problems defined in  \cite[Lemma 3.3]{SLL22}. Consequently, the limit in \eqref{lambdal} can be given explicitly. We will not dwell into more details and simply mention that using \cite[Theorem 3.4]{SLL22} one can extend Theorem \ref{th1} to sequences of multi-indices of the form
\[
  \Lambda^{\vec{v}} := \{\n +\vec{v}: \n \in \Lambda\},
\]
where $\Lambda \subset (\ZZ_+^m)^*$ verifies \eqref{const_ray} and $\vec{v} \in \ZZ_+^m$ is fixed.

\subsection{Rate of convergence of ML Hermite-Pad\'e approximants}

In \cite{lysov20}, the author obtains the logarithmic asymptotics of the linear forms $\mathcal{A}_{\n,j}$, $j=0,\ldots,m-1,$ and with it  estimates the rate of convergence of the  rational functions $\frac{a_{\n,j}}{a_{\n,m}}$, $j=0,\ldots,m-1,$ as $|\n| \to \infty$, along ray sequences of multi-indices. For straight ray sequences and Nikishin systems with Szeg\H{o}'s condition, using the results on strong asymptotics, we are able to give the exact rate of convergence.

\begin{corollary}
	\label{convaprox}
	Under the assumptions of Theorem \ref{th1}, we have
\begin{equation}
	\label{rate3}
	\lim_{\n\in\Lambda} \frac{\varepsilon_{\n,m}K_{\n,m-1}^2\Phi_m^{2|\n|}(z)}{\Phi_{m-1}^{\eta_{\n,m-1}}(z)}\left(\frac{a_{\n,j}}{a_{\n,m}}- \widehat{s}_{m,j+1}\right)(z) = \frac{\mathsf{G}_{m-1}(z)\mathsf{G}_m^{2}(\infty)}{\mathsf{G}_m^{2}(z)\mathsf{G}_{m-1}(\infty)}\frac{(-1)^{m-1} \widehat{s}_{m-1,j+1}(z)}{\sqrt{(z-a_m)(z-b_m)}},
\end{equation}
uniformly on compact subsets of $\overline{\CC}\setminus \cup_{i=j+1}^m \Delta_{i}$.
\end{corollary}

\begin{proof} Taking into account \cite[Lemma 2.1]{LSJ19} with $\mathcal{L}_j=\mathcal{A}_{\n,j}$, $j=0,1,\ldots,m-2$ and $r=m-1$, we have (for all points where the expression is meaningful)
\[
\mathcal{A}_{\n,j} + \sum_{k=j+1}^{m-1} (-1)^{k-j} \widehat{s}_{k,j+1}\mathcal{A}_{\n,k} = (-1)^j(a_{\n,j}-a_{\n,m}\widehat{s}_{m,j+1}).
\]
Though \cite{LSJ19} is dedicated to a special case of ML Hermite-Pad\'e approximation the forms $\mathcal{L}_j$ have a general character and that lemma is also applicable in the present situation.
Dividing the above equality by $Q_{\n,m} = a_{\n,m}$, we obtain
\[
\frac{\mathcal{A}_{\n,j}}{Q_{\n,m}} + \sum_{k=j+1}^{m-1} (-1)^{k-j} \widehat{s}_{k,j+1}\frac{\mathcal{A}_{\n,k}}{Q_{\n,m}} = (-1)^j\left(\frac{a_{\n,j}}{a_{\n,m}}- \widehat{s}_{m,j+1}\right).
\]
This is equivalent to
\begin{equation}
	\label{rate1}
\frac{\mathcal{A}_{\n,m-1}}{Q_{\n,m}} \left(\frac{\mathcal{A}_{\n,j}}{\mathcal{A}_{\n,m-1}} + \sum_{k=j+1}^{m-1} (-1)^{k-j} \widehat{s}_{k,j+1}\frac{\mathcal{A}_{\n,k}}{\mathcal{A}_{\n,m-1}} \right) = (-1)^j\left(\frac{a_{\n,j}}{a_{\n,m}}- \widehat{s}_{m,j+1}\right).
\end{equation}

From formula (1.8) in \cite[Proposition 1.2]{lysov20}, it follows that the ratios $\mathcal{A}_{\n,k}/\mathcal{A}_{\n,m-1}$, $k=j,\ldots,m-2$, converge uniformly to zero on compact subsets of $\overline{\CC}\setminus \cup_{i=j+1}^m \Delta_{i}$;  consequently,
\begin{equation}
	\label{cerococAnj}
\lim_{\n \in \Lambda}  \frac{\mathcal{A}_{\n,j}}{\mathcal{A}_{\n,m-1}} + \sum_{k=j+1}^{m-1} (-1)^{k-j} \widehat{s}_{k,j+1}\frac{\mathcal{A}_{\n,k}}{\mathcal{A}_{\n,m-1}} = (-1)^{m-1-j} \widehat{s}_{m-1,j+1},
\end{equation}
on compact subsets of $\overline{\CC}\setminus \cup_{i=j+1}^m \Delta_{i}$. On the other hand, from Corollary \ref{asym_forms} and \eqref{limfund***} we have
\begin{equation}
	\label{ratio_Am_1Qm}
	\lim_{n\in\Lambda} \frac{\varepsilon_{\n,m}K_{\n,m-1}^2\Phi_m^{2|\n|}(z)}{\Phi_{m-1}^{\eta_{\n,m-1}}(z)} \frac{\mathcal{A}_{\n,m-1}(z)}{Q_{\n,m}(z)} = \frac{\mathsf{G}_{m-1}(z)\mathsf{G}_m^{2}(\infty)}{\mathsf{G}_m^{2}(z)\mathsf{G}_{m-1}(\infty)}\frac{1}{\sqrt{(z-a_m)(z-b_m)}},
\end{equation}
uniformly on compact subsets of $\overline{\CC}\setminus(\Delta_{m-1}\cup\Delta_m)$.

Therefore, from \eqref{cerococAnj} and \eqref{ratio_Am_1Qm}, we get
\[
	\lim_{\n\in\Lambda} \frac{\varepsilon_{\n,m}K_{\n,m-1}^2\Phi_m^{2|\n|}(z)}{\Phi_{m-1}^{\eta_{\n,m-1}}(z)} \frac{\mathcal{A}_{\n,m-1}(z)}{Q_{\n,m}(z)} \left(\frac{\mathcal{A}_{\n,j}}{\mathcal{A}_{\n,m-1}} + \sum_{k=j+1}^{m-1} (-1)^{k-j} \widehat{s}_{k,j+1}\frac{\mathcal{A}_{\n,k}}{\mathcal{A}_{\n,m-1}} \right)(z) =\]
\[ \frac{\mathsf{G}_{m-1}(z)\mathsf{G}_m^{2}(\infty)}{\mathsf{G}_m^{2}(z)\mathsf{G}_{m-1}(\infty)}\frac{(-1)^{m-1-j} \widehat{s}_{m-1,j+1}(z)}{\sqrt{(z-a_m)(z-b_m)}},
\]
on compact subsets of $\overline{\CC}\setminus \cup_{i=j+1}^m$ which together with \eqref{rate1} give us \eqref{rate3}.
\end{proof}

\subsection{Strong asymptotics of Cauchy biorthogonal polynomials}  In this subsection, we apply Theorem \ref{th1} to obtain the strong asymptotics of Cauchy biorthogonal polynomials.

As above consider the Nikishin system $\mathcal{N}(\sigma_1,\ldots,\sigma_m)$ where $\supp \sigma_k \subset \Delta_k$ and the intervals $\Delta_k$ are bounded. Define the kernel function
\[
  K(x_1,x_m) = \int_{\Delta_2}\int_{\Delta_3}\cdots \int_{\Delta_{m-1} }
		\frac{\, \D \sigma_{m-1}(x_{m-1})\cdots \, \D\sigma_3(x_3)\D\sigma_2(x_2)   }
		{(x_{m-1}-x_m)(x_{m-2}-x_{m-1})\cdots\,(x_2-x_3)(x_1-x_2)  }.
\]
When $m=2$, take
\[ K(x_1,x_2) \equiv 1.
\]
 There exist two sequences of monic polynomials $(P_n), (Q_n), n \in \mathbb{Z}_+,$ such that for each $n$, $\deg(P_n) = \deg(Q_n) = n$ and
\[
  \int_{\Delta_1\times \Delta_m} P_k(x_1) K(x_1,x_m) Q_n(x_m) \,\mbox{d} \sigma_1(x_1) \,\mbox{d} \sigma_m(x_m)  = C_n \delta_{k,n},\qquad C_n \neq 0.
\]
As usual, $\delta_{k,n} = 0$, $k \neq n$, $\delta_{n,n} = 1$. The polynomials $P_n$ and $Q_n$ are called Cauchy biorthogonal polynomials with respect to $(\sigma_1,\ldots,\sigma_m)$.

When $m=2$ biorthogonal polynomials appear in the analysis of the two matrix model \cite{Ber,BGS13} and
in the search of discrete solutions of the Degasperis-Procesi equation \cite{Bertola:CBOPs} through a Hermite-Pad\'e approximation problem for two discrete measures.  In \cite{BGS13} a result on the strong asymptotics of Cauchy biorthogonal polynomials was obtained for a pair of Laguerre type weights with unbounded support and exponential decay at infinity. Using Theorem \ref{th1} we can give the strong asymptotics of general Cauchy biorthogonal polynomials with respect to a system of measures with compact support.

In \cite[Lemma 2.4]{ULS20} it was proved that the $n$-th biorthogonal polynomial $Q_n$ with respect to $(\sigma_1,\ldots,\sigma_m)$ coincides with the ML Hermite-Pad\'e polynomial $a_{\n,m}$  of the Nikishin system $\mathcal{N}(\sigma_1,\ldots,\sigma_m)$ associated with the multi-index $\n = (n,0,\ldots,0)$. Because of the symmetry of the kernel it turns out that the biorthogonal polynomial $P_n$ coincides with the corresponding $a_{\n,m}$ constructed for the reversed Nikishin system $\mathcal{N}(\sigma_m,\ldots,\sigma_1)$ and the multi-index $\n = (n,0, \ldots,0)$ (see \cite[Subsection 2.3]{ULS20}. Let us focus on $Q_n$. From Theorem \ref{th1}, we obtain

\begin{corollary}
	\label{cauchy}
	Assume that $\vec{\sigma} \in S(\vec{\Delta})$. Then
 \[
   \lim_{n \to \infty} \frac{Q_{n}(z)}{\Phi_m^{n}(z)} = \frac{\mathsf{G}_m(z)}{\mathsf{G}_m(\infty)},
\]
where $\Phi_m$ is given by \eqref{compar} for the vector equilibrium problem on the system of intervals $\vec{\Delta}$ with interaction matrix
 \[
 \left(\begin{array}{r r r r r}
 1      & -1/2 &  0 & \cdots &  0 \\
-1/2      &  1 & -1/2 & \ddots &   0 \\
 0      & -1/2 &  1 & \ddots &   0 \\
 \vdots &  \ddots  &  \ddots  & \ddots &    \vdots \\
 0      &  0 &  0 & \cdots &  1
\end{array}
\right)_{m\times m}
\]
and $\mathsf{G}_m$ is the last component of the fixed point of the operator
$\vec{T}_{\vec{w}}$ for the vector weight \eqref{weight:dir}
 where $\tilde{h}_m\equiv 1$ and the functions $\tilde{h}_j$, $j=1,\ldots,m-1$ are given by \eqref{def:htil}.
\end{corollary}


\begin{thebibliography}{99}

    \bibitem{sasha89} A.I. Aptekarev. Asymptotics of simultaneous orthogonal polynomials in the Angelesco case.  {Math. USSR Sb.} \textbf{64} (1989), 57--84.

    \bibitem{sasha99} A.I. Aptekarev. Strong asymptotics of multiply orthogonal polynomials for Nikishin systems.  {Sb. Math.} \textbf{190} (1999), 631--669.


    \bibitem{bel} M. Bello Hern\'andez, G. L\'opez Lagomasino and J. Mínguez Ceniceros.   Fourier-Pad\'e approximants for Angelesco systems. {Constr. Approx.} \textbf{26} (2007), 339--359.

    \bibitem{Ber} M.~Bertola. Two matrix models and biorthogonal polynomials. The Oxford handbook of random matrix theory, 310-328. Oxford Univ. Press, Oxford, 2011.


    \bibitem{Bertola:CBOPs} M. Bertola, M. Gekhtman, and J. Szmigielski. Cauchy biorthogonal polynomials.  {J. Approx. Theory} \textbf{162} (2010), 832-867. 

    \bibitem{BGS13} M. Bertola, M. Gekhtman, and J. Szmigielski. Strong asymptotics for Cauchy biorthogonal polynomials with application to the Cauchy two-matrix model.  {J. Math. Physics.} \textbf{54} (2013), 043517. 

    \bibitem{lagober98} B. de la Calle Ysern and G. López Lagomasino. Strong asymptotics of orthogonal polynomials with varying measures and Hermite-Pad\' e approximation.  {J. Comp. Appl. Math.} \textbf{99} (1998), 91--103. 

    \bibitem{bernardo_lago1} B. de la Calle Ysern and G. López Lagomasino. Weak convergence of varying measures and Hermite-Padé orthogonal polynomials.  Const. Approx. \textbf{15} (1999), 553-575.

    \bibitem{DKL} P. Deift, T. Krieckerbauer, and K. T.-R. McLaughlin. New results on the equilibrium measure for logarithmic potentials in the presence of an external field.  {J. of Approx. Theory} \textbf{95} (1998), 388-475. 

    \bibitem{ULS20} U. Fidalgo Prieto, G. López Lagomasino, and S. Medina Peralta. Asymptotic of Cauchy biorthogonal polynomials.  {Mediterr. J. Math.} \textbf{17} (2020), 17:22. 



\bibitem{LGLago} L.G. González Ricardo and G. López Lagomasino. Strong asymptotics of Cauchy biorthogonal polynomials and orthogonal polynomials with varying measure. Const. Approx. on-line
https://doi.org/10.1007/s00365-022-09580-7.

       \bibitem{SLL22} L.G. González Ricardo, G. López Lagomasino, and S. Medina Peralta. On the convergence of multi-level Hermite-Padé approximants. arxiv 2106.11370.


    \bibitem{GG} A. Granas and J. Dugundji. Fixed Point Theory. Springer Monographs in Math.. Springer-Verlag, New York, 2003.

    \bibitem{KD} A.B.J. Kuijlaars and P. Dragnev. Equilibrium problems associated with fast decreasing polynomials. {Proc. Amer. Math. Soc.} \textbf{127} (1999), 1065--1074. 

    \bibitem{LSJ19} G. López Lagomasino, S. Medina Peralta, and J. Szmigielski. Mixed type Hermite-Padé approximation inspired by the Degasperis-Procesi equation.  {Adv. Math.} \textbf{349} (2019), 813--838. 

    \bibitem{lago21} G. López Lagomasino. An introduction to multiple orthogonal polynomials and Hermite-Padé approximation. In:  {Orthogonal Polynomials: Current Trends and Applications, Proc. of the 7-th EIBPOA Conference}. Ed. by F. Marcellán and E.J. Huertas. SEMA SIMAI Springer Series, 22. Springer-Verlag, Berlin, 2021. 

    \bibitem{JacLun05} H. Lundmark and J. Szmigielski. Multi-peakon solutions of the Degasperis–Procesi equation. {Int. Math. Res. Pap.} \textbf{2} (2005), 53--116. 

    \bibitem{lysov20} V.G. Lysov. Mixed type Hermite–Padé approximants for a Nikishin system. {Proc. Steklov Inst. Math.} \textbf{311} (2020), 199--213. 

    \bibitem{mahler} K. Mahler. Perfect systems. {Compos. Math.} \textbf{19} (1968), 95--116. 


    \bibitem{nikishin82} E.M. Nikishin. On simultaneous Padé approximants. {Math. USSR Sb.} \textbf{41} (1982), 409--425. 

    \bibitem{NS} E.M. Nikishin and V.N. Sorokin.  {Rational Approximation and Orthogonality}. 1st ed. Translations of Mathematical Monographs 92. Amer. Math. Soc., Providence, RI, 1991.

    \bibitem{ReedSimon} M. Reed and B. Simon. Functional Analysis  Vol. I. Methods of Modern Mathematical Physics, Academic Press, London, 1980.

 \bibitem{SaffT} E.B. Saff and V. Totik. {Logarithmic Potentials with External Fields}. 1st ed. Grundlehren der mathematischen Wissenschaften, 316. Springer-Verlag, Berlin, 1997.

    \bibitem{simon1} B. Simon.  {Orthogonal Polynomials on the Unit Circle. Parts 1 and 2}. 1st ed. Colloquium Publications, 54. Amer. Math. Soc., Providence, RI, 2005.

    \bibitem{simon2} B. Simon.  {Szeg\H{o}'s Theorem and its Descendants.} Princeton Univ. Press. Princeton, 2011.


    \bibitem{Stahl_Totik} H. Stahl and V. Totik.  General Orthogonal Polynomials. Encyclopedia of Mathematics and its Applications \textbf{43}. Cambridge University Press, Cambridge, UK, 1992.

    \bibitem{szego} G. Szeg\H{o}.  {Orthogonal Polynomials}. 4th ed. Colloquium Publications XXIII. Amer. Math. Soc., Providence, RI, 1975.

    \bibitem{widom} H. Widom. Extremal polynomials associated with a system of curves in the complex plane. Adv. in Math. \textbf{3} (1969), 127-232.
\end{thebibliography}
\end{document}